\documentclass[11pt]{article}
\usepackage{amsmath, amssymb,verbatim,appendix}
\usepackage[mathscr]{eucal}
\usepackage{amscd}
\usepackage{amsthm}
\usepackage{stmaryrd}
\usepackage{enumerate}
\usepackage{comment}

\usepackage[margin=1in]{geometry}

\AtBeginDocument{%
   \def\MR#1{}
}

\newtheorem{theorem}{Theorem}
\newtheorem{lemma}[theorem]{Lemma}
\newtheorem{corollary}[theorem]{Corollary}

\newtheorem{remark}[theorem]{Remark}

\newtheorem{conjecture}[theorem]{Conjecture}
\newtheorem{definition}[theorem]{Definition}

\def\SOP{\operatorname{SOP}}
\def\Th{\operatorname{Th}}

\def\Forb{\operatorname{Forb}}

\title{
Discrete metric spaces: 
structure, enumeration, and $0$-$1$ laws
}

\author{
Dhruv Mubayi \footnote{Research supported in part by NSF Grant DMS 1300138}\\
University of Illinois at Chicago\\
mubayi@uic.edu
\and
Caroline Terry\\
University of Illinois at Chicago\\
cterry3@uic.edu
}

\begin{document}

\maketitle

\begin{abstract}
Fix an integer $r\geq 3$.  We consider metric spaces on $n$ points such that the distance between any two points lies in $\{1,\ldots, r\}$.  Our main result describes their approximate structure for large $n$.  As a consequence, we show that the number of these metric spaces is 
$$
\Big\lceil \frac{r+1}{2} \Big\rceil ^{{n\choose 2} + o(n^2)}.
$$
Related results in the continuous setting have recently been proved by Kozma, Meyerovitch, Peled, and Samotij \cite{KMPS}.
When $r$ is even, our structural characterization is more precise, and implies that almost all such metric spaces have all distances at least $r/2$. 
 As an easy consequence, when $r$ is even
 we improve the error term above from $o(n^2)$ to $o(1)$, and also 
 show a labeled first-order $0$-$1$ law in the language $\mathcal{L}_r$, consisting of $r$ binary relations, one for each element of $[r]$.  In particular, we show the almost sure theory $T$ is the theory of the Fra\"{i}ss\'{e} limit of the class of all finite simple complete edge-colored graphs with edge colors in $\{r/2,\ldots, r\}$.   
 
Our work can be viewed as an extension of a long line of research in extremal combinatorics to the colored setting, as well as an addition to the collection of known structures that admit  logical $0$-$1$ laws. 
\end{abstract}

%************************************************************************
\section{Introduction}\label{intro}

%************************************************************************
\setcounter{theorem}{0}
\numberwithin{theorem}{section}

Given integers $n,r\geq 3$, define $M_r(n)$ to the the set of all metric spaces with underlying set $[n]:=\{1, \ldots, n\}$ and distances in $\{1,\ldots, r\}$.  The goal of this paper is to investigate the approximate structure of most elements of $M_r(n)$ for fixed $r$ and large $n$, and in the case when $r$ is even, to prove that $M_r(n)$ has a labeled first-order $0$-$1$ law.

\subsection{Background} 
A graph is a set equipped with a symmetric irreflexive binary relation.  Given $n\in \mathbb{N}$ and  a collection $\mathcal{H}$ of graphs, let $\Forb_n(\mathcal{H})$ denote the set of graphs with vertex set $[n]$ which do not contain any element of $\mathcal{H}$ as a subgraph.  There is a long line of research in extremal combinatorics which investigates the structural properties of graphs in $\Forb_n(\mathcal{H})$ for various $\mathcal{H}$.  One of the first such results is due to Erd\H{o}s, Kleitman, and Rothschild \cite{EKR}, which states that if $\mathcal{H}=\{K_3\}$, then almost all graphs in $\Forb_n(\mathcal{H})$ are bipartite.  More precisely, if $B(n)$ is the set of bipartite graphs on $[n]$, then 
$$
\lim_{n\rightarrow \infty} \frac{|\Forb_n(\{K_3\})|}{|B(n)|} =1.
$$
In \cite{KPR}, Kolaitis, Promel, and Rothschild extend this result to the case when $\mathcal{H}=\{K_l\}$ for integers $l\geq 3$, showing that almost all $K_{l}$-free graphs are $(l-1)$-partite.  These fundamental combinatorial results have been extended and generalized in numerous ways.  For instance, in the graph setting, \cite{BBS3, BS, LR, MS} contain similar results about $\Forb_n(\mathcal{H})$ for specific collections $\mathcal{H}$, and \cite{BBS, BBS2, EFR, HPS, PS} contain results which apply to $\Forb_n(\mathcal{H})$ for $\mathcal{H}$ satisfying general properties.   Results of this spirit for other types of structures include, for example, \cite{BG, BGP, KR} for partial orders, \cite{OKTZ, Robinson, St} for directed graphs, and \cite{BM, BM2, PersonSchacht} for hypergraphs.

In some cases, the structural information obtained about $\Forb_n(\mathcal{H})$ from such investigation is enough to prove a labeled first-order $0$-$1$ law, which we now define.  Suppose $\mathcal{L}$ is a finite first-order language and $F=\bigcup_{n\in \mathbb{N}} F_n$, where $F_n$ is a set of $\mathcal{L}$-structures with underlying set $[n]$.  For each $\mathcal{L}$-sentence $\psi$, set $\mu_n(\psi)$ to be the proportion of elements in $F_n$ which satisfy $\psi$.  Then the \emph{asymptotic probability} of $\psi$ is $\mu(\psi)=\lim_{n\rightarrow \infty} \mu_n(\psi)$ (if it exists).  We say $F$ has a \emph{labeled first-order limit law} if for each $\mathcal{L}$-sentence $\psi$, $\mu(\psi)$ exists, and we say $F$ has a \emph{labeled first-order $0$-$1$ law} if moreover, for each $\mathcal{L}$-sentence $\psi$, we have $\mu(\psi)\in \{0,1\}$.  The \emph{almost sure theory} of $F$ is the set of $\mathcal{L}$-sentences $\psi$ such that $\mu(\psi)=1$.  In \cite{KPR}, Kolaitis, Pr\"{o}mel, and Rothschild use the structural information they obtain about $\Forb(\{K_{l}\})=\bigcup_{n\in \mathbb{N}}\Forb_n(\{K_l\})$ for $l\geq 3$ to show that each such family has a labeled first-order $0$-$1$ law in the language of graphs and to give an axiomatization of its almost sure theory.

Given a set $X$, let ${X\choose 2}=\{Y\subseteq X: |Y|=2\}$ and $2^X=\{Y: Y \subset X\}$.  An \emph{$r$-graph} $G$ is a pair $(V,c)$, where $V$ is a (vertex) set, and  $c:{V\choose 2}\rightarrow 2^{[r]}$; we call $G$  a \emph{simple complete $r$-graph} if $|c(xy)|=1$ for all $xy \in {V \choose 2}$.  Elements of $M_r(n)$ are naturally viewed as simple complete $r$-graphs by  assigning edge colors corresponding to distances.  Given  a set $\mathcal{H}$ of $r$-graphs, let $\Forb^r_n(\mathcal{H})$ be the set of simple complete $r$-graphs with vertex set $[n]$ which contain no element of $\mathcal{H}$ as a substructure.  By taking $\mathcal{H}$ to be the set of simple complete $r$-graphs on three vertices which contain violations of the triangle inequality, we see that $M_r(n) = \Forb^r_n(\mathcal{H})$. In this way, we can view $M_r(n)$ as an edge-colored analogue of $\Forb_n(\mathcal{H})$. This analogy suggests that one could prove similar results as in \cite{EKR} and \cite{KPR} about $M_r(n)$.  In this paper we show that this is indeed the case, utilizing techniques from graph theory to describe the approximate structure of most elements of $M_r(n)$ for large $n$. 

We may view elements of $M_r(n)$ as first-order structures in the language $\mathcal{L}_r$ consisting of $r$ binary predicates, one for each edge color.  In this setting, as a corollary of our structural results, we are able to prove in the case when $r$ is even, that there is a labeled first-order $0$-$1$ law for $M_r=\bigcup_{n\in \mathbb{N}} M_r(n)$ and to give an axiomatization of its almost sure theory.  In this paper, we consider only $r\geq 3$ for the following reason.  There is no way to violate the triangle inequality using distances in $\{1,2\}$, so $M_2(n)$ consists of the set of \emph{all} simple complete $2$-graphs.  This means that given a pair $x,y$ of distinct elements of $[n]$, the distance between $x$ and $y$ is equal to $1$ in exactly half of the elements of $M_2(n)$.  For each $G\in M_2(n)$, associate a graph $\mathcal{G}$ with vertex set $[n]$ such that for each $x,y\in [n]$, there is an edge between $x$ and $y$ in $\mathcal{G}$ if and only if the distance between $x$ and $y$ is equal to $1$ in $G$.  Under this association, we see that $M_2(n)$ behaves exactly like the random graph $G(n,1/2)$, the structural properties of which have been studied extensively (see \cite{Bo}), and which is known to have a labeled first-order $0$-$1$ law \cite{Fagin, GKLT}.

The results of this paper may be of interest  to both combinatorialists and model theorists.  From the combinatorial perspective, our work appears to be the first extension of the classical enumeration results in extremal graph theory to the edge-colored setting.  The proofs of our main results will rely on a stability theorem which is proved using a multi-color version of the Szemer\'{e}di regularity lemma \cite{AxMa}.  While our proof techniques bear some resemblance to the classical results in \cite{EFR, EKR, KR}, we need several new ideas that are motivated by work on weighted Tur\'{a}n-type problems \cite{FK}.  Our contributions also add to existing results that study metric spaces as combinatorial objects \cite{CRT, KMPS, Ma2, Ma1}.  In particular, \cite{Ma1} and \cite{KMPS} address questions similar to ours in the continuous setting.   In \cite{Ma1}, Mascioni investigates the following problem.  Given an integer $n$ and a fixed set $X$ of $n$ points, if we assign i.i.d. uniform real numbers in $[0,1]$ to the elements of ${X\choose 2}$, what is the probability we get a metric space?  It is shown in \cite{Ma1} that this probability $p$ satisfies 
\begin{equation}\label{mas}
\left(\frac{1}{2}\right)^{n\choose 2} \le p \le \left(\frac{1}{2}\right)^{\lfloor n/2\rfloor} \left(\frac{2}{3}\right)^{\lfloor n/2\rfloor (\lceil n/2 \rceil -2)},\end{equation}
 where the lower bound is obtained by noting that any assignment of distances from $[\frac{1}{2},1]$ yields a metric space.  In more recent work, Kozma, Meyerovitch, Peled, and Samotij~\cite{KMPS} identify the set of metric spaces on $[n]$ having all distances in $[0,1]$ with elements in the cube $[0,1]^{{n\choose 2}}$.  Let $M_n$ be the subset of $[0,1]^{{n\choose 2}}$ which corresponds to the set of metric spaces on $[n]$.  Then \cite{KMPS} shows that  there are constants $c$ and $C$ such that 
\begin{eqnarray}\label{KMPS}
\frac{1}{2}+\frac{c}{\sqrt{n}} \leq (\text{vol} M_n)^{1/{n\choose 2}} \leq \frac{1}{2}+\frac{C}{n^c}.
\end{eqnarray}
%Since $[\frac12,1]^{{n\choose 2}} \subset M_n$ and $\text{vol}([\frac12,1]^{{n\choose 2}})^{1/{n\choose 2}}=1/2$, 
They also prove that with high probability, all distances are between $1/2-n^{-c}$ and $1$. The upper bound in  (\ref{KMPS}) implies that the probability $p$ in (\ref{mas}) approaches the lower bound as $n\rightarrow \infty$.  
Given a fixed even $r\geq 4$, our results about $M_r(n)$ can be translated into results about metric spaces on $[n]$ with all distances  in $\{\frac{1}{r},\ldots, \frac{r-1}{r}, 1\}$.  In this setting, our Theorem \ref{aathm} says that almost all such  metric spaces (as $n\rightarrow \infty$) have all of their distances in $[\frac12, 1]$ therefore capturing a similar phenomenon as the results of \cite{KMPS}  (for odd $r$ the situation is slightly more complicated). If it were possible to generalize our results to the setting where $r \rightarrow \infty$ and $n$ is fixed, then they  could apply to the continuous setting.
 
From the model theory perspective, we provide a new example which may aid in understanding further why some classes of structures have labeled first-order limit laws and others do not.  There has been much investigation into finding sufficient conditions for when a class of finite structures has various types of logical limit laws.  One type of sufficient conditions, first introduced by Compton in \cite{Compton, ComptonII}, requires that the number of structures of size $n$ does not grow too quickly as $n\rightarrow \infty$.  The theorems in \cite{Compton, ComptonII} and  various extensions of them (for instance \cite{Bell, BurrisYeats}) provide a large number of examples of logical limit laws.  However, there are many examples of families with logical limit laws which fail these conditions on the growth rate of the family, for instance $\Forb(\{K_l\})$ for $l\geq 3$ fails these conditions but has a labeled first-order $0$-$1$ law \cite{KPR}.  $M_r$ also fails these conditions for all $r\geq 3$.  In \cite{Koponen2} Koponen presents conditions which cover more known examples.  In particular, it is shown in \cite{Koponen2} that the family of almost $l$-partite graphs for $l\geq 2$ has a logical limit law.  Koponen combines this with the main result of \cite{HPS} to prove the existence of logical limit laws for $\Forb(\{\mathcal{H}\})$ when $\mathcal{H}$ is a complete $(l+1)$-partite graph with parts of sizes $1,s_1,\ldots, s_l$, for some $1\leq s_1\leq \ldots \leq s_l$.  When $s_1=\ldots =s_l$, $\mathcal{H}=K_{l+1}$, so this generalizes the $0$-$1$ law proved in \cite{KPR} for $\Forb(\{K_l\})$,  $l\geq 3$.  More results on logical limit laws for various families of graphs appear in \cite{HK, HMNT, Koponen3, Lynch}.  However, these results do not apply to $M_r$, as elements of $M_r$ are not graphs.  

In \cite{Koponen}, Koponen studies the asymptotic probability of extension axioms in families of structures in finite relational languages satisfying certain general requirements.  This generality allows the results to be applied to structures other than graphs.  For example, Koponen combines results of \cite{Koponen} with the main results of \cite{BM} and \cite{PersonSchacht} to show certain families of hypergraphs with forbidden configurations have labeled first-order $0$-$1$ laws (see Example 10.7 of \cite{Koponen}).  Another paper which studies logical limit laws for more general languages is \cite{AK} by Ahlman and Koponen, which focuses on families of structures in finite relational languages which satisfy certain colorability requirements and have an underlying pregeometry.  While none of these results apply directly to $M_r$, a result of \cite{Koponen} does imply that a subfamily $C_r$ of $M_r$, (which will be defined later) has a labeled  first-order $0$-$1$ law.  Our results will show that when $r$ is even, almost all elements of $M_r$ are in $C_r$, which will yield that  $M_r$ has a labeled first-order $0$-$1$ law.  Therefore, this paper provides the combinatorial argument required to reduce the existence of a labeled first-order $0$-$1$ law for $M_r$ to the existence of one for $C_r$, while the fact that $C_r$ has a labeled first-order $0$-$1$ law follows from known results, and is in fact very easy to prove directly.  Part of the motivation for this work is the idea that having more examples of logical limit laws in languages other that of graphs, and seeing the techniques used to prove them, will  improve our general understanding of when a family of finite structures has a logical limit law.

\subsection{New Results}\label{newresults} In this section we state the results in this paper.  First we give some necessary definitions and notation.  Given  positive integers $r,s$ and a set $X$, set $[r]=\{1,\ldots, r\}$, ${X\choose s}=\{Y\subseteq X: |Y|=s\}$, and $2^X=\{Y : Y\subseteq X\}$.  Recall that an \emph{$r$-graph} $G$ is a pair $(V,c)$, where $V$ is a set, and $c:{V\choose 2}\rightarrow 2^{[r]}$.  We call $V$ the \emph{vertex set} of $G$ and $c$ the \emph{coloring} of $G$.  In the case when $|c(e)|\leq 1$ for every $e\in {V\choose 2}$, we say that $(V,c)$ is \emph{simple}, and when $c(e)\neq \emptyset$ for each $e\in {V\choose 2}$, we say $G$ is \emph{complete}.  Given integers $r,n\geq 3$, we consider $M_r(n)$ as the set of simple complete $r$-graphs $([n],c)$ satisfying the triangle inequality, i.e, for every three pairwise distinct elements $x,y,z$ of $[n]$ we have 
$$c(x,z)\leq c(x,y)+c(y,z).$$  Given a set $X$ and $\{x,y\}\in {X\choose 2}$, we will write $xy$ to mean $\{x,y\}$.  Given integers $i<j$, set $[i,j]=\{i,i+1, \ldots, j\}$.

\begin{definition}
For an even integer $r\geq 4$ and any integer $n$, let $C_r(n)$ be the set of  all simple complete $r$-graphs $G=([n],c)$ such that $c(e)\subset [\frac{r}{2}, r ]$ for all $e\in {[n]\choose 2}$.
\end{definition}

When $r$ is even, there is no way to violate the triangle inequality using distances in $[\frac{r}{2}, r ]$, so $C_r(n) \subset M_r(n)$.
The strongest structural result we will prove (Theorem \ref{aathm} below) says that when $r\geq 4$ is even, almost all elements in $M_r(n)$ are in  $C_r(n)$.

\begin{theorem}\label{aathm}
Let $r\geq 4$ be an even integer.  Then there is $\beta>0$ and $M>0$ such that for all $n\geq M$, 
$$
|C_r(n)|\geq |M_r(n)|(1-2^{-\beta n}).
$$
\end{theorem}
When $r$ is even, $|C_r(n)|=(\frac{r}{2}+1)^{n\choose 2}$. Therefore Theorem \ref{aathm} yields that when $r$ is even,
$$
\left(\frac{r}{2}+1 \right)^{{n\choose 2}}\leq |M_r(n)| \leq \left( \frac{1}{1-2^{-\beta n}}\right)\left(\frac{r}{2}+1\right)^{{n\choose 2}}
$$
for some positive $\beta$ and sufficiently large $n$.  We obtain the following Corollary.
\begin{corollary}\label{evencor}
Let $r\geq 4$ be an even integer.  Then $|M_r(n)|=(\frac{r}{2}+1)^{{n\choose 2}+o(1)}$.
\end{corollary}

When $r$ is odd, we still obtain a result on the approximate structure of most elements of $M_r(n)$ (Theorem \ref{deltaclosethm} below), however the situation in this case is more complicated.

\begin{definition} Let $r\geq 3$ be an odd integer.  Define $C_r(n)$ to be the the set of simple complete $r$-graphs $G=([n],c)$ such that there is a partition $V_1\cup \ldots\cup  V_t$ of $[n]$ and for every $xy\in {[n]\choose 2}$, 
\[
c(xy) \subset \begin{cases}
[ \frac{r-1}{2}, r-1]& \text{if } xy\in {V_i\choose 2}\textnormal{ for some }i\in [t] \\
[\frac{r+1}{2}, r] & \text{if } x\in V_i, y\in V_j \textnormal{ for some }i\neq j\in [t].
\end{cases}
\]
\end{definition}
\noindent It is  easy  to see that for $r$ odd, $C_r(n)\subset M_r(n)$.
Given $\delta>0$, two $r$-graphs $G=(V,c)$ and $G'=(V,c')$ with the same vertex set $V$ are \emph{$\delta$-close} if $|\{e\in {V\choose 2}: c(e)\neq c'(e)\}|\leq \delta |V|^2$.  Set 
$$
C_r^{\delta}(n)=\{G\in M_r(n): \textnormal{there is }G' \in C_r(n)\textnormal{ such that }G\textnormal{ and }G'\textnormal{ are }\delta\textnormal{-close} \}. 
$$
We now state our structure theorem which holds for all $r\geq 3$.  Informally, it states that most members of $M_r(n)$ are in $C_r^{\delta}(n)$ for small $\delta$ and $n$ large enough depending on $\delta$.

\begin{theorem}\label{deltaclosethm}
Let $r\geq 3$ be an integer.   Then for all $\delta>0$, there exists an $M$ and $\beta >0$ such that $n>M$ implies
$$
\frac{|M_r(n)\setminus C_r^{\delta}(n)|}{|M_r(n)|}\leq \frac{|M_r(n)\setminus C_r^{\delta}(n)|}{\lceil \frac{r+1}{2} \rceil^{n\choose 2}}\leq 2^{-\beta n^2}.
$$
\end{theorem}

\begin{corollary}\label{oddcor}
Let $r\geq3$ be an integer.  Then $|M_r(n)|=\lceil \frac{r+1}{2}\rceil^{{n\choose 2}+o(n^2)}$.
\end{corollary}

We will prove as a consequence of Theorem \ref{aathm} that, when $r$ is even, $M_r=\bigcup_{n\in \mathbb{N}}M_r(n)$ has a labeled first-order $0$-$1$ law in the language $\mathcal{L}_r$ consisting of $r$ binary relation symbols, in the process giving an axiomatization of its almost sure theory.  

\begin{theorem}\label{01thm}
Let $r\geq4$ be an even integer and define $\mathcal{L}_r = \{R_1,\ldots, R_r\}$ where each $R_i$ is a binary relation symbol.  Given $n\in \mathbb{N}$, consider elements $G=([n],c)\in M_r(n)$ as $\mathcal{L}_r$-structures by interpreting for each $(x,y)\in [n]^2$, $R^G_i(x,y)\Leftrightarrow xy\in {[n]\choose 2}$ and $c(xy)=\{i\}$.  Then $M_r$ has a labeled first-order $0$-$1$ law.
\end{theorem}

When $r$ is odd, the error term in Corollary~\ref{oddcor} cannot be improved from $o(n^2)$ to $O(n)$, and moreover, Theorem~\ref{aathm} does not hold (See Section 7 for a detailed discussion).  
This leads us to make the following conjecture.

\begin{conjecture}\label{conjecture}
Let $r\geq 3$ be an odd integer and consider elements of $M_r(n)$ as $\mathcal{L}_r$-structures as in Theorem \ref{01thm}.  Then $M_r=\bigcup_{n\in \mathbb{N}} M_r(n)$ has a labeled first-order limit law, but does not have a labeled first-order $0$-$1$ law.
\end{conjecture}

%******************************************************************************************************************************
\subsection{Notation and outline} \label{notationandoutline}
%******************************************************************************************************************************
Throughout the paper, we will omit floors and ceilings where they are unimportant to the argument.  Let $r\geq 3$ be an integer and let $G$ be an $r$-graph.  We will write $V(G)$ to denote the vertex set of $G$ and $c^G$ to denote its coloring.  For simplicity of notation we set $E(G)={V(G)\choose 2}$, and for subsets $X,Y\subseteq V(G)$, set $E(X,Y)=\{xy \in E(G): x\in X, y\in Y\}$, and $E(X)=E(X,X)$.  Given a simple complete $r$-graph $G$, we define $d^G: E(G)\rightarrow [r]$ to be the function sending $xy\in E(G)$ to the unique $i\in [r]$ such that $c^G(xy) =\{i\}$.  We will sometimes also wish to discuss graphs, meaning a set equipped with a single binary, symmetric, irreflexive relation.  In order to avoid confusion, graphs with be denoted by $\mathcal{G}=(\mathcal{V},\mathcal{E})$, where $\mathcal{V}$ is the vertex set of $\mathcal{G}$ and $\mathcal{E}\subseteq {\mathcal{V}\choose 2}$ is the edge set of $\mathcal{G}$.  Given a graph $\mathcal{G} = (\mathcal{V},\mathcal{E})$ and $v\in \mathcal{V}$, we will write $\mathcal{DEG}(v)=|\{u: uv\in \mathcal{E}\}|$.

By a \emph{violating triple} we will mean a tuple $(i,j,k)\in \mathbb{N}^3$ such that $|i-j|\leq k\leq i+j$ is false.  By a \emph{violating triangle}, we will mean an $r$-graph $H$ such that $V(H)=\{x,y,z\}$, and for some violating triple $(i,j,k)$, $i\in c^H(xy)$, $j\in c^H(yz)$, and $k\in c^H(xz)$.  Define a \emph{metric} $r$-graph to a be an $r$-graph $G=(V,c)$ which contains no violating triangles.  Given two $r$-graphs $H$ and $G$, with $|V(G)|=n$ and $V(H)=\{y_1,\ldots, y_m\}$, we say $G$ \emph{omits} $H$ if for all $(x_1,\ldots, x_m)\in V(G)^m$, there is $1\leq s<t\leq m$ such that $c^G(x_s x_t) \neq c^H(y_s y_t)$.  When $G$ does not omit $H$, we say $G$ \emph{contains a copy of $H$}.  Given two finite $r$-graphs $G$ and $G'$ with $V(G)=V(G')$, set
$$
\Delta(G,G') =\{xy\in E(G): c^G(xy)\neq c^{G'}(xy)\}.
$$
In this notation, given $\delta>0$, $G$ and $G'$ are $\delta$-close if $|\Delta(G,G')|\leq \delta |V(G)|^2$.  Given a set of finite $r$-graphs $S$ and a finite $r$-graph $G$, say that $G$ is \emph{$\delta$-close to} $S$ if $G$ is $\delta$-close to some element of $S$.   
Given $r\geq 3$, set 
$$
m(r)=\Bigg\lceil \frac{r+1}{2}\Bigg\rceil.
$$
A subset $A\subseteq [r]$ is called a \emph{metric set} if $A^3$ contains no violating triples.  Note that when $r$ is even, $[\frac{r}{2},r]$ is a metric subset of $[r]$ of size $m(r)$.  When $r$ is odd, both $[\frac{r-1}{2}, r-1]$ and $[\frac{r+1}{2}, r]$ are metric subsets of $[r]$ of size $m(r)$.  As remarked earlier, any $r$-graph meeting the requirements in the definition of $C_r(n)$ is already in $M_r(n)$.  In particular, $C_r(n)$ contains all simple complete metric $r$-graphs with distances in $[m(r),r]$, therefore $|C_r(n)|\geq m(r)^{n\choose 2}$.  These observations yield the following fact we will use throughout the paper.

\begin{remark}\label{lowerbound}
Let $n, r\geq 3$ be integers. Then 
$$
|M_r(n)|\geq |C_r(n)|\geq m(r)^{n\choose 2},
$$
and if $r$ is even, then $|C_r(n)|=m(r)^{{n\choose 2}}$.
\end{remark}

We now give an outline of the paper.   In section \ref{01law} we introduce the notion of a labeled  first-order $0$-$1$ law, and prove as a consequence of Theorem \ref{aathm} that Theorem \ref{01thm} is true, i.e. when $r\geq4$ is an even integer, $M_r$ has a labeled first-order $0$-$1$ law in the language consisting of $r$ binary predicates.  In section \ref{enumsection} we prove Corollary \ref{oddcor}, which provides an asymptotic enumeration of $M_r(n)$ as a consequence of Theorem \ref{deltaclosethm}.  In section \ref{stabilitysection} we provide preliminaries and notation regarding a multi-color version of Szemer\'{e}di's regularity lemma, then we prove Theorem \ref{stabthm}, which is a stability result needed to prove Theorem \ref{deltaclosethm}.  In section \ref{provingtheorem3} we prove Theorem \ref{deltaclosethm}, and in section \ref{provingtheorem1} we prove Theorem \ref{aathm}.  Finally, in section \ref{concludingremarks}, we explain why Corollary \ref{evencor} and Theorem \ref{aathm} do not hold when $r$ is odd, then discuss open questions concerning $M_r(n)$ when $r$ is odd.

%************************************************************************
\section{Proof of logical $0$-$1$ law}\label{01law}

%************************************************************************
\setcounter{theorem}{0}
\numberwithin{theorem}{section}

In this section we assume Theorem \ref{aathm} and prove Theorem \ref{01thm}, which says that for even integers $r\geq 4$, the family $M_r=\bigcup_{n\in \mathbb{N}}M_r(n)$ has a labeled first-order $0$-$1$ law in the language $\mathcal{L}_r$ consisting of $r$ binary relation symbols. The outline of the argument is as follows.  Theorem \ref{aathm} allows us to reduce Theorem \ref{01thm} to showing the existence of a labeled first-order $0$-$1$ for the subfamily $C_r=\bigcup_{n\in \mathbb{N}} C_r(n)$. The existence of a labeled first-order $0$-$1$ law for $C_r$ follows from a standard argument.  In particular, it follows from a theorem in \cite{Koponen} which generalizes the method in \cite{Fagin}.  We assume familiarity with the theory of Fra\"{i}ss\'{e} limits.  For an introduction to this subject, see chapter 7 of \cite{Hodges}.  For a survey on logical $0$-$1$ laws see \cite{Winkler}.  We begin with the required terminology concerning $0$-$1$ laws.

\begin{definition}\label{01def}
Let $\mathcal{L}$ be a finite first-order language.  For each $n$, suppose $V_n$ is a set of $\mathcal{L}$-structures on $[n]$, and  $V=\bigcup_{i\in \mathbb{N}}V_i$.  
\begin{enumerate}
\item $\mu_n^V: V_n\rightarrow [0,1]$ is the probability measure defined by setting $\mu^V_n(G) = \frac{1}{|V_n|}$ for each $G\in V_n$.
\item Given a first-order $\mathcal{L}$-sentence $\psi$, set $\mu^V_n(\psi) = \mu_n^V(\{ G\in V_n: G\models \psi\})$ and $\mu^V(\psi) = \lim_{n\rightarrow \infty} \mu^V_n(\psi)$.  When $\mu^V(\psi)$ exists, it is called the \emph{labeled asymptotic probability of $\psi$}.
\item The \emph{almost sure theory of $V$} is $T^V_{as} = \{\psi: \psi \textnormal{ is an }\mathcal{L}\textnormal{-sentence and }\lim_{n\rightarrow \infty} \mu^V_n(\psi) =1\}$.  
\item $V$ has a \emph{labeled first-order $0$-$1$ law} if for each first-order $\mathcal{L}$-sentence $\psi$, $\mu^V(\psi)$ exists and is $0$ or $1$.  
\end{enumerate}
\end{definition}  
It is straightforward to show that $V$ has a labeled first-order $0$-$1$ law if and only if $T^V_{as}$ is a complete, consistent theory with infinite models.  

Fix an even integer $r\geq 4$ for the rest of the section.  Define $\mathcal{L}_r=\{R_1(x,y),\ldots, R_r(x,y)\}$, where each $R_i(x,y)$ is a binary relation symbol.  Given an $r$-graph $G$, make $G$ into an $\mathcal{L}_r$-structure by interpreting for all $(x,y)\in V(G)^2$, 
$$
R^G_i(x,y) \Leftrightarrow xy\in E(G)\textnormal{ and }i\in c^G(xy).
$$
From here on, all $r$-graphs will be considered as $\mathcal{L}_r$-structures in this way. We now prove that as a consequence of Theorem \ref{aathm}, $M_r$ has a labeled first-order $0$-$1$ law if and only if $C_r$ does.

\begin{lemma}\label{M_riffC_r}
For all $\mathcal{L}_r$-sentences $\psi$, if $\mu^{C_r}(\psi)$ exists, then $\mu^{M_r}(\psi)$ exists, and moreover, $\mu^{C_r}(\psi) =\mu^{M_r}(\psi)$.
\end{lemma}
\begin{proof}
Assume $\mu^{C_r}(\psi)$ exists.  For all $n$, 
\begin{eqnarray}\label{ineq}
\mu^{M_r}_n(\psi) = \frac{|\{G\in M_r(n)\setminus C_r(n): G\models \psi\}|}{|M_r(n)|} +\frac{|\{G\in C_r(n): G\models \psi\}|}{|M_r(n)| }.
\end{eqnarray}
By Theorem \ref{aathm}, there is $\beta >0$ such that for sufficiently large $n$, 
$$
|M_r(n)\setminus C_r(n)|\leq 2^{-\beta n} |M_r(n)|\textnormal{ and }|C_r(n)|\leq |M_r(n)|\leq (1+2^{-\beta n})|C_r(n)|,
$$
where the second inequality is because for all $n$, $C_r(n)\subseteq M_r(n)$.  Thus for sufficiently large $n$, 
$$
\frac{|\{G\in C_r(n): G\models \psi\}|}{|C_r(n)|(1+2^{-\beta n}) }\leq \frac{|\{G\in C_r(n): G\models \psi\}|}{|M_r(n)|} \leq  \frac{|\{G\in C_r(n): G\models \psi\}|}{|C_r(n)| }.
$$
and
$$
\frac{|\{G\in M_r(n)\setminus C_r(n): G\models \psi\}|}{|M_r(n)|} \leq 2^{-\beta n}.
$$
Therefore 
$$
\lim_{n\rightarrow \infty} \frac{|\{G\in M_r(n)\setminus C_r(n): G\models \psi\}|}{|M_r(n)|} =0
$$
and
$$
\lim_{n\rightarrow \infty} \frac{|\{G\in C_r(n): G\models \psi\}|}{|M_r(n)| } = \lim_{n\rightarrow \infty} \frac{|\{G\in C_r(n): G\models \psi\}|}{|C_r(n)| } =\mu^{C_r}(\psi).
$$
Combining these with (\ref{ineq}) yields that $\mu^{M_r}(\psi) = \mu^{C_r}(\psi)$.
\end{proof}

Lemma \ref{M_riffC_r} implies that to prove Theorem \ref{01thm}, it suffices to show $C_r$ has a labeled first-order $0$-$1$ law, and further, that an axiomatization of $T_{as}^{C_r}$ will also axiomatize $T_{as}^{M_r}$. Towards stating the axiomatization of $T_{as}^{C_r}$, we now fix some notation.  Fix an integer $k\geq 2$.  Given $A\in M_r(k)$, write $x_1\ldots x_k \equiv A$ as short hand for the $\mathcal{L}_r$-formula which says that sending $x_i\mapsto i$ makes $x_1\ldots x_k$ isomorphic to $A$.  Explicitly we mean the formula $\psi(x_1,\ldots, x_k)$ given by
$$
\bigwedge_{1\leq i<j\leq k} \Bigg( R_{d^A(i,j)}(x_i,x_j) \wedge \bigwedge_{s\neq d^A(i,j)} \neg R_s(x_i,x_j)\Bigg).
$$
Given $A\in M_r(k)$ and $A'\in M_r(k+1)$, write $A\prec A'$ to denote that for all $ij\in {[k]\choose 2}$, $d^A(ij)=d^{A'}(ij)$.  Given such a pair $A\prec A'$, let $\sigma_{A'/A}$ be the following sentence:
$$
\forall x_1\ldots \forall x_k ((x_1\ldots x_k \equiv A) \rightarrow \exists y (x_1\ldots x_k y\equiv A')).
$$
Sentences of this form are called \emph{extension axioms}.  Let $T_1$ be a set of $\mathcal{L}_r$-sentences axiomatizing an infinite metric space with distances all in $[\frac{r}{2},r]$, 
\begin{align*}
T_2 &= \bigcup_{k\in \mathbb{N}} \{ \sigma_{A'/A}: A\in C_r(k), A'\in C_r(k+1), A\prec A' \},\text{ and }\\
T &=T_1\cup T_2.
\end{align*}
$T$ will be the set of sentences axiomatizing $T_{as}^{C_r}=T_{as}^{M_r}$.
\bigskip

\noindent{\bf Proof of Theorem \ref{01thm}}.  By the arguments above, it suffices to show $C_r$ has a labeled first-order $0$-$1$ law.  Let $\mathcal{C}_r$ be the class of $\mathcal{L}_r$-structures obtained by closing $C_r$ under isomorphism.   That $\mathcal{C}_r$ is a Fra\"{i}ss\'{e} class is straightforward to see.  For the sake of completeness we verify that $\mathcal{C}_r$ has the amalgamation property.  Given $X,Y \in \mathcal{C}_r$, an isometry $f:X\rightarrow Y$ is an injective map from $V(X)$ into $V(Y)$ such that for all $xy\in E(X)$, $d^X(x,y) = d^Y(f(x),f(y))$.  Suppose $A, B, C\in \mathcal{C}_r$ and $f:C\rightarrow A$, $g:C\rightarrow B$ are isometries.  Without loss of generality, assume that $f$ and $g$ are inclusion maps and $V(A)\cap V(B) = V(C)$.  To verify the amalgamation property, we want to find $D\in \mathcal{C}_r$ and isometries $h:A\rightarrow D$ and $s:B\rightarrow D$ such that for all $c\in V(C)$, $s(c)=h(c)$.  We do this by setting $V(D)=V(A)\cup V(B)$ and for $xy\in {V(D)\choose 2}$, setting
\begin{eqnarray} \label{AP}
d^D(x,y) = \begin{cases}
d^A(x,y) & \textnormal{ if } xy\in E(A),\\
d^B(x,y) & \textnormal{ if } xy \in E(B)\setminus E(A),\\
r & \textnormal{ if } x\in (V(A)\setminus V(C)), y\in (V(B)\setminus V(C)).
\end{cases} 
\end{eqnarray}
$D$ is a simple complete $r$-graph with $d^D(x,y)\in [\frac{r}{2},r]$ for all $xy\in E(D)$, so $D\in \mathcal{C}_r$.   Define $h:A\rightarrow D$ and $s:B\rightarrow D$ to be the inclusion maps.  Then for all $c\in V(C)$, $h(c)=s(c)=c$, as desired, and $\mathcal{C}_r$ has the amalgamation property.  Note that we could have chosen any color in $[\frac{r}{2},r]$ to assign the edges in the third case of (\ref{AP}), as there are no forbidden configurations in $\mathcal{C}_r$.  We leave the rest of the verification that $\mathcal{C}_r$ is a Fra\"{i}ss\'{e} class to the reader.

Let $FL(\mathcal{C}_r)$ be the Fra\"{i}ss\'{e} limit of $\mathcal{C}_r$ and make $FL(\mathcal{C}_r)$ into an $\mathcal{L}_r$-structure by interpreting, for each $(x,y)\in FL(\mathcal{C}_r)^2$, $R_i(x,y)$ if and only if $d^{\mathcal{C}_r}(x,y)=i$.  It is a standard exercise to see that $FL(\mathcal{C}_r)\models T$ and further that $T$ axiomatizes $\Th(FL(\mathcal{C}_r))$.  Therefore $T$ is a complete, consistent $\mathcal{L}_r$-theory, so to show $C_r$ has a labeled first-order $0$-$1$ law, it suffices to show that for each $\psi \in T$, $\mu^{C_r}(\psi)=1$.  For $\psi\in T_1$, this is obvious.  Because there are no forbidden configurations in $C_r$, a straightforward counting argument shows that for $\psi \in T_2$, $\mu^{C_r}(\neg \psi) =0$, and therefore $\mu^{C_r}(\psi) =1$.  An example of such an argument applied to graphs is the proof of Lemma 2.4.3 of \cite{Dave}.  The proof in our case is only slightly more complicated, so we omit it.  We also point out that this fact (that for all $\psi\in T_2$, $\mu^{C_r}(\psi)=1$) follows directly from a much more general result, Theorem 3.15 of \cite{Koponen}.  Because this theorem is much more powerful than what our example requires, we leave it to the interested reader to verify it applies to $C_r$ and $\psi \in T_2$.
\qed
\bigskip

We end this section by showing that while there is a Fra\"{i}ss\'{e} limit naturally associated to $M_r$, its theory is very different from the almost sure theory we obtain from $M_r$.  Let $\mathcal{M}_r$ be the class of finite metric spaces obtained by closing $M_r$ under isomorphism, that is, $\mathcal{M}_r$ is the class of all finite metric spaces with distances all in $[r]$.  It is well known that $\mathcal{M}_r$ is a Fra\"{i}ss\'{e} class.  For instance, this is a simple case of general results contained in \cite{DLPS}, which tell us when, given $S\subseteq \mathbb{R}$, the class of finite metric spaces with distances all in $S$ forms a Fra\"{i}ss\'{e} class. For completeness we verify the amalgamation property for our case, that is, when $S=[r]$.

Suppose $A, B, C\in \mathcal{M}_r$ and $f:C\rightarrow A$, $g:C\rightarrow B$ are isometries.  Without loss of generality, assume that $f$ and $g$ are inclusion maps and $V(A)\cap V(B) = V(C)$.  To verify the amalgamation property, we want to find $D\in \mathcal{M}_r$ and isometries $h:A\rightarrow D$ and $s:B\rightarrow D$ such that for all $c\in V(C)$, $s(c)=h(c)$.  Given $s,t\in [r]$, let $t\dotplus s= \min\{ r, t+s\}$. Set $V(D)=V(A)\cup V(B)$ and for $xy\in {V(D)\choose 2}$, set
\begin{eqnarray} \label{APM_r}
d^D(x,y) = \begin{cases}
d^A(x,y) & \textnormal{ if } xy\in E(A),\\
d^B(x,y) & \textnormal{ if } xy \in E(B)\setminus E(A),\\
\max\{ d^A(x,c)\dotplus d^B(c,y): c\in V(C)\} & \textnormal{ if } x\in (V(A)\setminus V(C)), y\in (V(B)\setminus V(C)).
\end{cases} 
\end{eqnarray}
We leave it to the reader to verify that the assigned distances do not violate the triangle inequality, and therefore, that $D$ is in $\mathcal{M}_r$.  Define $h:A\rightarrow D$ and $s:B\rightarrow D$ to be the inclusion maps.  Then for all $c\in V(C)$, $h(c)=s(c)=c$, as desired, and  $\mathcal{M}_r$ has the amalgamation property.  Note that unlike in the proof of the amalgamation property for $\mathcal{C}_r$, the distance in the third line of (\ref{APM_r}) must be chosen carefully, as there are many forbidden configurations in $\mathcal{M}_r$.

Let $FL(\mathcal{M}_r)$ be the Fra\"{i}ss\'{e} limit of $\mathcal{M}_r$.  It is a standard exercise that the theory of $FL(\mathcal{M}_r)$ is axiomatized by the axioms for an infinite metric space with distances all in $[r]$ and the collection of all extension axioms of the form $\sigma_{A'/A}$ for some $A\in M_r(k)$, $A'\in M_r(k+1)$ with $A\prec A'$, and $k\geq 0$.  We can see now that $\Th(FL(\mathcal{M}_r))$ and $\Th(FL(\mathcal{C}_r))$ are different.  For instance, let $\psi$ be the sentence 
$$
\exists x\exists y R_1(x,y).
$$
Then $\psi \in Th(FL(\mathcal{M}_r))$, while clearly $\Th(FL(\mathcal{C}_r))\models \neg \psi$.  Model theoretically, $Th(FL(\mathcal{C}_r))$ is simple (in the sense of Definition 7.2.1 in \cite{TZ}).  This can be seen by adapting the argument used to prove the theory of the random graph is simple, as $\mathcal{C}_r$ is just an edge-colored version of the random graph (see Corollary 7.3.14 in \cite{TZ} for a proof that the theory of the random graph is simple).  On the other hand, a straightforward adjustment of the construction in Theorem 5.5(b) of \cite{CT} shows that $Th(FL(\mathcal{M}_r))$ has \emph{the $r$-strong order property} ($\SOP_r$), a measure of the complexity of a first-order theory defined in \cite{Sh}.  It is shown in \cite{Sh} that for all $n\geq 3$, a theory with $\SOP_n$ is not simple.  In sum, when $r\geq 4$ is even, we have a family of labeled finite structures, $M_r$, associated to two theories which differ in model theoretic complexity: 
\begin{enumerate}[$\bullet$]
\item $\Th(FL(\mathcal{M}_r))$ where $\mathcal{M}_r$ is obtained by closing $M_r$ under isomorphism.  This theory has $\SOP_r$ (and therefore is not simple).
\item $T_{as}^{M_r}=T_{as}^{C_r}=\Th(FL(\mathcal{C}))$, where $C_r\subseteq M_r$ is a special subfamily, and $\mathcal{C}_r$ is obtained by closing $C_r$ under isomorphism.  This theory is simple.
\end{enumerate}

%************************************************************************
\section{Asymptotic Enumeration}\label{enumsection}

%************************************************************************
\setcounter{theorem}{0}
\numberwithin{theorem}{section}

In this section we assume Theorem \ref{deltaclosethm} and prove Corollary \ref{oddcor}, which asymptotically enumerates $M_r(n)$ for all $r\geq 3$.  Recall that for all integers $r\geq 3$, $m(r)=\lceil \frac{r+1}{2} \rceil$.  
\bigskip

\noindent{\bf Proof of Corollary \ref{oddcor}}.  
Fix an integer $r\geq 3$.  All logs will be base $m(r)$ unless otherwise stated.  Remark \ref{lowerbound} implies that $|M_r(n)|\geq m(r)^{{n\choose 2}}$, so it suffices to show that for all $0<\gamma<1$, there is $M$ such that $n>M$ implies $
|M_r(n)|<m(r)^{{n\choose 2}+\gamma n^2}$.

Fix $0<\gamma<1$.  Let $H(x)=-x\log_2 x - (1-x)\log_2 (1-x)$ and recall that $H(x)\rightarrow 0$ as $x\rightarrow 0$ and ${n\choose xn}\leq 2^{H(x)n}$ for all $n\in \mathbb{N}$ and $0<x\leq \frac{1}{2}$.  Choose $\delta>0$ small enough so that
$$
(H(\delta)+\delta)\log2 +\delta \log r <\frac{\gamma}{4}.
$$
Theorem \ref{deltaclosethm} implies there exists a $\beta=\beta(\delta)>0$ and $M_1=M_1(\delta)$ such that $n>M_1$ implies
$$
|M_r(n)\setminus C^{\delta}_r(n)|\leq 2^{-\beta n^2}m(r)^{{n\choose 2}}.
$$
Choose $M>M_1$ large enough so that $n>M$ implies $\frac{\gamma}{4} n^2+n\log n < \frac{\gamma}{2} n^2$ and $\frac{\gamma}{2} n^2 +\log 2\leq \gamma n^2$.  We now assume $n>M$ and bound the size of $C^{\delta}_r(n)$.  All elements $G\in C^{\delta}_r(n)$ can be constructed as follows:
\begin{enumerate}[$\bullet$]
\item Choose an element of $G'\in C_r(n)$.  There are $|C_r(n)|$ ways to do this.  If $r$ is even, then $|C_r(n)|=m(r)^{n\choose 2}$.  If $r$ is odd, we must find an upper bound for $|C_r(n)|$.  When $r$ is odd, we can construct any element of $C_r(n)$ by first choosing a partition of $[n]$, then assigning a color to each edge in a way compatible with the partition.  There are at most $n^n m(r)^{{n\choose 2}}$ ways to do this.  
\item Choose at most $\delta n^2$ edges to be in $\Delta(G,G')$.  There are at most ${n^2 \choose \delta n^2}2^{\delta n^2}\leq 2^{(H(\delta)+\delta)n^2}$ ways to do this.
\item Assign a color to each edge in $\Delta(G,G')$.  There are at most $r^{\delta n^2}$ ways to do this.
\end{enumerate}
Thus
$$
|C^{\delta}_r(n)|\leq n^n m(r)^{{n\choose 2}}2^{(H(\delta)+\delta)n^2}r^{\delta n^2}= m(r)^{{n\choose 2}+ n^2((H(\delta)+\delta)\log2 +\delta \log r)+n\log n}.
$$
By our assumptions on $\delta$ and $M$, this is at most $m(r)^{{n\choose 2} + \frac{\gamma}{4} n^2+n\log n} < m(r)^{{n\choose 2} +\frac{\gamma}{2}n^2}$.
Therefore, since $M_r(n) = (M_r(n)\setminus C_r^{\delta}(n))\cup C^{\delta}_r(n)$ we have
$$
|M_r(n)| \leq m(r)^{{n\choose 2} -n^2\beta \log 2} + m(r)^{{n\choose 2}+\frac{\gamma}{2}n^2}\leq 2m(r)^{{n\choose 2}+\frac{\gamma}{2}n^2} =m(r)^{{n\choose 2}+\frac{\gamma}{2}n^2+\log 2} \leq m(r)^{{n\choose 2}+\gamma n^2},
$$
where the last inequality is by the choice of $M$.
\qed

%************************************************************************
\section{Stability Theorem}\label{stabilitysection}

%************************************************************************
\setcounter{theorem}{0}
\numberwithin{theorem}{section}

In this section we prove a stability theorem which implies that for all integers $r\geq 3$, for large enough $n$, if $G\in M_r(n)$ has close to the maximal number of different distances occurring between its vertices, then it is structurally close to an element of $C_r(n)$.  This is a crucial step in the proofs of Theorems \ref{aathm} and \ref{deltaclosethm}.  Before proceeding further, we require some definitions and notation.

%******************************************************************************************************************************
\subsection{Regularity Lemmas and Preliminaries}
%******************************************************************************************************************************
In this section we state a version of Szemer\'{e}di's Regularity Lemma which applies to $r$-graphs.  We will also prove easy consequences of this for our situation.  

\begin{definition} Let $r\geq 3$ be an integer. Fix a finite $r$-graph $G$ and disjoint subsets $X,Y\subseteq V(G)$.
\begin{enumerate}
\item Suppose $\mathcal{A}=\{A_1,\ldots, A_m\}$ is a partition of $V(G)$.  $\mathcal{A}$ is an \emph{equipartition} if $||A_i|-|A_j||\leq 1$ for all $i\neq j$, and the \emph{order} of $\mathcal{A}$ is $m$.  A \emph{refinement} of $\mathcal{A}$ is a partition $\mathcal{B}=\{B_1,\ldots, B_k\}$ such that for each $i\in [k]$, there is $j\in [m]$ such that $B_i\subseteq A_j$.
\item For $l \in [r]$, set
\begin{align*}
e^G_{l}(X,Y) &:=|\{ xy \in E(X,Y): l\in c^G(xy)\}|\textnormal{ and }\\
\rho^G_{l}(X,Y) &:= \frac{e_{l}(X,Y)}{|X||Y|}.
\end{align*}
\item The \emph{density vector} of $(X,Y)$ in $G$ is $(\rho^G_1,\ldots, \rho^G_r)$ where $\rho^G_i=\rho^G_i(X,Y)$.
\item $(X,Y)$ is \emph{$\epsilon$-regular for $G$} if for all $X' \subseteq X$ and $Y'\subseteq Y$ with $|X'|\geq \epsilon |X|$ and $|Y'|\geq \epsilon |Y|$, for all $l \in [r]$, 
$$
|\rho^G_{l}(X,Y) - \rho^G_{l}(X',Y')|\leq \epsilon.
$$
\item A partition $\mathcal{B}=\{B_1,\ldots, B_k\}$ of $V(G)$ is called \emph{$\epsilon$-regular for $G$} if it is an equipartition of $V(G)$, and for all but at most $\epsilon k^2$ of the pairs $ij\in {[k]\choose 2}$, $(B_i,B_j)$ is $\epsilon$-regular for $G$.
\end{enumerate}
\end{definition}

We now state the multi-color version of the Szemeredi Regularity Lemma and one of its corollaries we will use in this paper.  Both results appear in \cite{AxMa}.  
\begin{theorem}\label{Szlem} (Regularity Lemma)
Fix an integer $r\geq 2$.  For every $\epsilon > 0$ and positive integer $m$, there is an integer $CM=CM(m,\epsilon)$ such that if $G$ is a finite $r$-graph with at least $CM$ vertices, and $\mathcal{A}$ is an equipartition of $G$ of order $m$, then there $k$ such that $m\leq k\leq CM$ and a refinement $\mathcal{B}$ of $\mathcal{A}$ of order $k$ which is $\epsilon$-regular for $G$.
\end{theorem}

\begin{theorem} (Embedding Lemma) \label {EL}
Fix an integer $r\geq 2$.  For every $0<d<1$ and $k\in \mathbb{N}\setminus\{0\}$, there is $\gamma= \gamma_{el}(d, k)\leq d$ and $\delta= \delta_{el}(d, k)$ such that the following holds.  Suppose that $H$ and $G$ are $r$-graphs and $V(H)=\{v_1,\ldots, v_k\}$.  Suppose $V_1,\ldots, V_k$ are pairwise disjoint subsets of $V(G)$ such that for every $ij\in {[k]\choose 2}$, $(V_i, V_j)$ is $\gamma$-regular for $G$, and for each $l\in [r]$, $l \in c^H(v_iv_j)$ implies $\rho^G_{l}(V_i, V_j)\geq d$.   Then there are at least $\delta \prod_{i=1}^k |V_i|$ $k$-tuples $(w_1,\ldots, w_k)\in V_1 \times \cdots \times V_k$ such that for each $ij\in {[k]\choose 2}$, $c^H(v_iv_j)\subseteq c^G(w_iw_j)$.
\end{theorem}

We will apply these theorems to what are called \emph{reduced $r$-graphs}, which we define below.  Recall that a metric $r$-graph is an $r$-graph with no violating triangles.
\begin{definition}
Let $r\geq 2$ be an integer, $G$ a finite $r$-graph, and $0<\eta \leq d\leq 1$.
\begin{enumerate}
\item Suppose $\mathcal{P}=\{V_1,\ldots, V_t\}$ is an $\eta$-regular partition for $G$.  Let $R(G,\mathcal{P}, d)$ be the $r$-graph $R$ with vertex set $[t]$ such that $s\in c^R(ij)$ if and only if $(V_i,V_j)$ is $\eta$-regular for $G$ and $\rho_s(V_i,V_j)\geq d$.  We say $R$ is a \emph{reduced $r$-graph} obtained from $G$ with parameters $\eta$ and $d$.
\item Let $\tilde{M}_r(t)$ be the set of metric $r$-graphs on $[t]$ and set
\begin{align*}
Q_{\eta,d,t}(G) &= \{ R(G,\mathcal{P},d) : \mathcal{P}\textnormal{ is an }\eta\textnormal{-regular equipartition for } G\textnormal{ and }\mathcal{P}\textnormal{ has order }t\},\textnormal{ and }\\
Q_{\eta, d}(G) &= \bigcup_{t=\frac{1}{\eta}}^{CM(\frac{1}{\eta}, \eta)} Q_{\eta,d,t}(G).
\end{align*}
\end{enumerate}
\end{definition}

We emphasize that the difference between $\tilde{M}_r(t)$ and $M_r(t)$ is that $r$-graphs in $\tilde{M}_r(t)$ need not be simple and need not be complete.  The following two lemmas will be needed.
\begin{lemma}\label{elcor}
Let $r\geq2$ be an integer, $0<d<1$, $0<\gamma \leq \gamma_{el}(d, 3)$, and $\delta\leq \delta_{el}(d,3)$.  Let $(i,j,k)\in [r]^3$ be a violating triple.  Suppose $G\in M_r(n)$ and $V_1, V_2, V_3\subseteq V(G)$ are pairwise disjoint and $\gamma$-regular for $G$ with $\delta|V_1||V_2||V_3| \geq 1$.  If $\{X,Y,Z\}=\{V_1,V_2,V_3\}$, then 
\begin{equation} \label{min}\min \{\rho^G_i(X,Y), \rho^G_j(Y,Z), \rho^G_k(X,Z)\} <d.
\end{equation}
\end{lemma}
\begin{proof}
Suppose for contradiction that $\{X,Y,Z\}=\{V_1,V_2,V_3\}$ and (\ref{min}) fails.  By Theorem \ref{EL} there exists at least $\delta |V_1||V_2||V_3|\geq 1$ tuples $(x,y,z)\in X\times Y\times Z$ such that $i\in c^G(xy)$, $j\in c^G(yz)$ and $k\in c^G(xz)$.  But now $\{x,y,z\}$ is a violating triangle in $G$, a contradiction.
\end{proof}

\begin{lemma}\label{Qnonempty}
Let $0<d<1$ and $0<\eta\leq \gamma_{el}(d,3)$.  There is an $M$ such that $n>M$ implies that for all $G\in M_r(n)$, $\emptyset \neq Q_{\eta,d}(G) \subseteq \bigcup_{t=\frac{1}{\eta}}^{CM(\frac{1}{\eta}, \eta)}\tilde{M}_r(t)$.  In other words, any reduced $r$-graph obtained from $G$ with parameters $d$ and $\eta$ omits all violating triangles.
\end{lemma}
\begin{proof}
Let $M= \frac{2CM(\frac{1}{\eta},\eta)}{\delta_{el}(d,3)^{\frac{1}{3}}}$.  Suppose $n>M$ and $G\in M_r(n)$.  As $n>CM(\frac{1}{\eta},\eta)$, there is $t$ with $\frac{1}{\eta}\leq t\leq CM(\frac{1}{\eta},\eta)$ and $\mathcal{P}=\{V_1,\ldots, V_t\}$ an $\eta$-regular partition for $G$.  Therefore $Q_{\eta,d,t}(G)\neq \emptyset$, so $Q_{\eta,d}(G)\neq \emptyset$.  Let $R=R(G,\mathcal{P},d)\in Q_{\eta,d,t}(G)$.  We will show that $R\in \tilde{M}_r(t)$.  Note that for all $V_i,V_j,V_k \in \mathcal{P}$, 
$$
\delta_{el}(d,3)|V_i||V_j||V_k| \geq \delta_{el}(d,3)\Bigg(\frac{n}{t}-1\Bigg)^3 >\delta_{el}(d,3)\Bigg(\frac{n}{2t}\Bigg)^3 \geq \delta_{el}(d,3)\frac{n^3}{8CM(\frac{1}{\eta},\eta)^3}\geq1,
$$
by assumption on $M$.  Thus by Lemma \ref{elcor}, $R$ contains no violating triangle, so $R\in \tilde{M}_r(t)$.
\end{proof}

We spend the rest of this section stating various definitions and facts we will need for our proofs.  We will work with the following subset $\tilde{C}_r(n)\subseteq \tilde{M}_r(n)$ which is an analogue of $C_r(n)\subseteq M_r(n)$.

\begin{definition} Let $r\geq 3$ be an integer.  Set $\tilde{C}_r(t)$ to be the the set of complete $r$-graphs $R$ with $V(R)=[t]$ such that
\begin{enumerate}[(i)]
\item if $r$ is even, then for all $xy\in E(R)$, $c^R(xy)=[\frac{r}{2}, r]$.
\item if $r$ is odd, then there is a partition $[t]=V_1\cup \ldots\cup  V_s$ such that for all $xy\in {[t]\choose 2}$, 
\[
c^R(xy) = \begin{cases}
[\frac{r-1}{2},r-1] & \text{if } xy\in {V_i\choose 2}\textnormal{ for some }i\in [s] \\
[\frac{r+1}{2}, r] & \text{if } xy\in E(V_i,V_j)\textnormal{ for some }i\neq j\in [s].
\end{cases}
\] 
\end{enumerate}
\end{definition}
Note that elements of $\tilde{C}_r(t)$ contain no violating triangles, so $\tilde{C}_r(t)\subseteq \tilde{M}_r(t)$.  The following weight function defined on metric $r$-graphs is crucial to our proof.

\begin{definition} Let $t\geq 2$ and $r\geq 3$ be integers and let $R\in \tilde{M}_r(t)$.  For $ij \in {[t]\choose 2}$, set 
$$
f^R(i,j) = \max\{ |c^R(ij)|,1\} \qquad \textnormal{ and }\qquad W(R) = \prod_{ij\in {[t]\choose 2}} f^R(i,j).
$$
\end{definition}
\noindent Note that for integers $r,t\geq 3$, any $r$-graph $R$ with $t$ vertices has $W(R)\leq r^{t\choose 2}$.  Recall that when $r$ is even $m(r) = |[\frac{r}{2}, r]|$ and when $r$ is odd, $m(r) = |[\frac{r-1}{2}, r-1]|=|[\frac{r+1}{2},r]|$, so for any integers $r,t\geq 3$, for all $R\in \tilde{C}_r(t)$ and $ij\in {[t]\choose 2}$, $f^R(i,j) = m(r)$, and thus $W(R)=m(r)^{{t\choose 2}}$.  

We now state a lemma which restricts how many colors we can assign to the edges of a triangle $\{i,j,k\}$ in an $r$-graph without creating a violating triangle.  The proof of this lemma is elementary but somewhat tedious, and for this reason is relegated to the Appendix.  

\begin{lemma}\label{sizelemma}
Fix an integer $r\geq 3$.  Let $A$, $B$, and $C$ be nonempty subsets of $[r]$ such that $|A|\geq |B|\geq |C|$, $|A|>m(r)$, and $|B|\geq m(r)$.  Set $x=|A|-m(r)$ and $y=|B|-m(r)$, and suppose
\[
|C|\geq 
\begin{cases}
\max \{m(r)-x-y,1\} & \text{if } r \text{ is even} \\
\max \{m(r)-x-y+2,1\} & \text{if } r \text{ is odd} .
\end{cases}
\]
Then there is a violating triple $(a,b,c)\in A\times B\times C$. 
\end{lemma}

A straightforward consequence of this is that  $m(r)$ is the largest size of a metric subset of $[r]$.   Another important consequence is the following.

\begin{corollary}\label{importantcor}
Let $r,t \geq 3$ be integers and let $R\in \tilde{M}_r(t)$.  Suppose $uv, vw, uw \in E(R)$, and $f^R(u,v) \geq f^R(v,w)>m(r)$.  Then $f^R(u,w)<m(r)$ and $\max\{ f^R(u,v)f^R(u,w),f^R(v,w)f^R(u,w)\}\leq m(r)^2-1$.
\end{corollary}
\begin{proof}
For $xy\in {[t]\choose 2}$, set $f(x,y)=f^R(x,y)$.  Given $A,B,C\subseteq [r]$ and $x,y\in [r]$, write $P(A,B,C,x,y)$ if $|A|\geq |B|\geq |C|$, $x=|A|-m(r)$, $y=|B|-m(r)$, $|A|>m(r)$ and $|B|\geq m(r)$.  Set $A=c^R(u,v)$, $B=c^R(v,w)$, $C=c^R(u,w)$, $x=|A|-m(r)$, and $y=|B|-m(r)$.  Note $|A|=f(u,v)$, $|B|=f(v,w)$, $|C|=f(u,w)$, and $|A|\geq |B|$ by assumption.  We show that $|A|\geq |B|\geq |C|$.  Suppose for a contradiction that $|C|>|B|$.  Let $z=|C|-m(r)$ and note our assumptions imply that either $P(A,C,B, x,z)$ or $P(C,A,B, z,x)$ holds.  In either case, $|B|> m(r)\geq m(r)-x-z+2$ implies by Lemma \ref{sizelemma} that there is a violating triple $(a,b,c)\in A\times B\times C$.  Now $\{u,v,w\}$ is a violating triangle in $R$, a contradiction.  Thus $|A|\geq |B|\geq |C|$.

Consequently, $P(A,B,C,x,y)$ holds, so if $|C|\geq m(r)-x-y+2$ were true, Lemma \ref{sizelemma} would imply that there is a violating triple $(a,b,c)\in A\times B\times C$, making $\{u,v,w\}$ a violating triangle in $R$, a contradiction.  Therefore, we must have $|C|<m(r)-x-y+2$.  Our assumptions imply that $x,y\geq 1$, so in fact, $|C|< m(r)$.  Further, we have shown that
$$
|B||C|=f(v,w) f(u,w)\leq (m(r)+y)(m(r)-x-y+1)\leq (m(r)+y)(m(r)-y)= m(r)^2-y^2 \leq m(r)^2-1,
$$
and
$$
|A||C|=f(u,v) f(u,w)\leq (m(r)+x)(m(r)-x-y+1)\leq (m(r)+x)(m(r)-x)= m(r)^2-x^2 \leq m(r)^2-1,
$$
as desired.\end{proof}

%******************************************************************************************************
\subsection{Two Lemmas}
%******************************************************************************************************
In this section, we prove two lemmas toward our stability result.  The first lemma bounds the size of $W(R)$ for $R\in \tilde{M}_r(t)$.  We will frequently use the following inequality which holds for all integers $r\geq 3$:
\begin{eqnarray}\label{m(r)^2-1>r}
m(r)^2 -1\geq r.
\end{eqnarray}

\begin{lemma}\label{boundW(R)}
Let $t,r\geq 3$ be integers and $R\in \tilde{M}_r(t)$.  Let $a_R=|\{ij \in E(R): f^R(i,j)>m(r)\}|$.  Then
$$
W(R) \leq m(r)^{{t\choose 2}+t+5} \Bigg( \frac{m(r)^2 -1}{m(r)^2}\Bigg)^{a_R}.
$$
\end{lemma}
\begin{proof}
Fix an integer $r\geq 3$.  Given an integer $t$ and $R\in \tilde{M}_r(t)$, set $g(R) =m(r)^{{t\choose 2}+t+5}( \frac{m(r)^2-1}{m(r)^2})^{a_R}$.  We proceed by induction on $t$.  Assume $t=3$ and fix $R\in \tilde{M}_r(t)$.  In this case $a_R\leq 3$, so $g(R) \geq m(r)^{5}(m(r)^2-1)^3$.  It is straightforward to verify that $r^3 \leq m(r)^{5}$, as $r\geq 3$.  Therefore, 
$$
W(R)\leq r^3 \leq m(r)^{5}(m(r)^2-1)^3 \leq g(R).
$$
Assume now that $t>3$ and the claim holds for all $t'$ with $3\leq t'<t$.  Fix $R\in \tilde{M}_r(t)$, set $a=a_R$, and for $xy\in {[t]\choose 2}$, set $f(x,y)=f^R(x,y)$.  If $a=0$ then $W(R)\leq m(r)^{t\choose 2}\leq g(R)$ trivially.  So assume $a>0$.

Choose $uv \in E(R)$ such that $f(u,v)$ is maximum, and note that $a>0$ implies $f(u,v)>m(r)$.  Define $R'$ to be the $r$-graph with $V(R')=[t]\setminus \{u,v\}$ and for each $xy\in E(R')$, $c^{R'}=c^R|_{V(R')}$.  Let $a' = a_{R'}$,
$$
Y=\{z\in V(R'): \max\{ f(u,z), f(v,z)\} >m(r)\},
$$
and set $s= |Y|$.  For all $z\in Y$, because $\max\{f(u,z),f(v,z)\}>m(r)$ and $f(u,v)>m(r)$, Corollary \ref{importantcor} implies $\min \{f(u,z),f(v,z)\}<m(r)$ and  $f(u,z)f(v,z)\leq m(r)^2-1$.  By the definition of $Y$, for all $z\in V(R')\setminus Y$, $\max\{f(u,z),f(v,z)\}\leq m(r)$, so $f(u,z)f(v,z)\leq m(r)^2$.  Combining these facts we have
\begin{align*}
W(R)&=W(R')f(u,v)\Bigg(\prod_{z\in Y} f(u,z)f(z,v)\Bigg)\Bigg(\prod_{z\notin Y} f(u,z)f(z,v)\Bigg)\\
&\leq W(R')f(u,v) (m(r)^2-1)^{s}m(r)^{2(t-2-s)} \leq W(R')r (m(r)^2-1)^{s}m(r)^{2(t-2-s)}.
\end{align*}
Using (\ref{m(r)^2-1>r}), we can upper bound this by
\begin{align*}
W(R') (m(r)^2-1)^{s+1}m(r)^{2(t-2-s)} = W(R') \Bigg( \frac{m(r)^2-1}{m(r)^2}\Bigg)^{s+1} m(r)^{2t-2}.
\end{align*}
By the induction hypothesis, this is at most
\begin{align*}
& m(r)^{{t-2\choose 2}+t-2+5} \Bigg( \frac{m(r)^2-1}{m(r)^2}\Bigg)^{a'}\Bigg( \frac{m(r)^2-1}{m(r)^2}\Bigg)^{s+1} m(r)^{2t-2} = m(r)^{{t\choose 2}+t+4} \Bigg( \frac{m(r)^2-1}{m(r)^2}\Bigg)^{a'+s+1}.
\end{align*}
Note that $a=a' + |\{zu: z\in Y$ and $f(u,z)>m(r)\}\cup \{vz: z\in Y$ and $f(v,z)>m(r)\}\cup \{uv\}|$.  Because for each $z\in Y$ exactly one of $f(u,z)$ or $f(v,z)$ is strictly greater than $m(r)$, this shows $a=a'+s+1$.  Therefore,
$$
W(R)\leq m(r)^{{t\choose 2}+t+4} \Bigg( \frac{m(r)^2-1}{m(r)^2}\Bigg)^{a} < g(R).
$$
This completes the proof.
\end{proof}

We now fix some notation.  Suppose $r\geq 3$ is an integer, $R$ is an $r$-graph, and $u\in V(R)$.  For $i\in [r]$, set
\begin{align*}
N^R_i(u) &= \{v\in V(R): i\in c^R(uv)\} \textnormal{ and }\\
\Gamma^R_i(u) &= \{v\in V(R): f^R(u,v) = i\}.
\end{align*}
Then define $\deg^R_i(u) = |N^R_i(u)|$ and $\mu^R_i(u)= |\Gamma^R_i(u)|$.  We now prove the second lemma.

\begin{lemma}\label{boundnonf(ij)=m(r)}
For every integer $r\geq 3$ there are $C_1$, $C_2$, $C_3$, depending only on $r$ such that for every $0<\epsilon<1$, there is $M$ such if $t>M$ the following holds.  Suppose $R\in \tilde{M}_r(t)$ with $W(R)>m(r)^{(1-\epsilon){t\choose 2}}$.  Let $a_R=|\{ij\in E(R) : f^R(i,j)>m(r)\}|$ and $b_R=|\{ij\in E(R): f^R(i,j)<m(r)\}|$.  Then 
\begin{enumerate}
\item $a_R\leq C_1\epsilon t^2$,
\item $b_R\leq C_2\epsilon t^2$, and
\item $|\{u : \mu^R_{m(r)}(u)< (1-\sqrt{\epsilon})(t-1)\}| \leq \sqrt{\epsilon}C_3 t$.
\end{enumerate}
\end{lemma}
%$c_1 = \frac{3}{4\log_{m(r)}(\frac{m(r)^2}{m(r)^2-1})} , 
%c_2= \Bigg( \frac{\frac{1}{2}+c_1\log_{m(r)} r}{1-\log_{m(r)}(m(r)-1)}\Bigg),
%c_3 = 3(c_1+\Bigg( \frac{\frac{1}{2}+c_1\log_{m(r)} r}{1-\log_{m(r)}(m(r)-1)}\Bigg) )
\begin{proof}
Let $r,t\geq 3$ be integers.  Fix $\epsilon>0$ and suppose $R\in \tilde{M}_r(t)$ is such that $W(R) >m(r)^{(1-\epsilon){t\choose 2}}$.  Set $a=a_R$ and $b=b_R$. All logs in this proof are base $m(r)$. Our assumptions and Lemma \ref{boundW(R)} imply $m(r)^{(1-\epsilon){t\choose 2}}< W(R)\leq m(r)^{{t\choose 2}+t+5}(\frac{m(r)^2-1}{m(r)^2})^a$.  Consequently, 
\begin{align}\label{boundnonf(ij)=m(r)a}
\Bigg(\frac{m(r)^2}{m(r)^2-1} \Bigg)^a< m(r)^{\epsilon{t\choose 2}+t+5}.
\end{align}
Suppose $M_1$ is large enough so that $t>M_1$ implies $t(1-\frac{\epsilon}{2}) + 5<\frac{\epsilon t^2}{4}$, and assume  $t>M_1$.  Taking log of both sides of (\ref{boundnonf(ij)=m(r)a}) we obtain
$$
a\log\Bigg(\frac{m(r)^2}{m(r)^2-1}\Bigg) \leq \epsilon{t\choose 2} +t+5< \frac{\epsilon}{2} t^2 +\frac{\epsilon}{4} t^2= \frac{3\epsilon t^2}{4},
$$
where the last inequality is by assumption on $M_1$.  Therefore $a\leq C_1 \epsilon t^2$, for appropriate choice of $C_1=C_1(r)$.   This proves (1).  For (2), note that by the definitions of $W(R)$, $a$, and $b$ we have 
$$
W(R) \leq (m(r)-1)^b r^a m(r)^{{t\choose 2}-a-b}.
$$
Thus our assumptions and part (1) imply that,
$$
m(r)^{(1-\epsilon){t\choose 2}}< (m(r)-1)^b r^{C_1\epsilon t^2} m(r)^{{t\choose 2}-a-b} \leq  (m(r)-1)^b r^{C_1\epsilon t^2} m(r)^{{t\choose 2}-b} = \Bigg(\frac{m(r)-1}{m(r)}\Bigg)^b r^{C_1\epsilon t^2}m(r)^{{t\choose 2}}.
$$
Consequently,
$$
\Bigg(\frac{m(r)}{m(r)-1}\Bigg)^b < m(r)^{\epsilon{t\choose 2}}r^{C_1\epsilon t^2}.
$$
Taking $\log$ of both sides, we obtain
$$
b\log \Bigg(\frac{m(r)}{m(r)-1}\Bigg) < \epsilon {t\choose 2} + C_1\epsilon t^2 \log r < \Bigg(\frac{1}{2} + C_1\log r\Bigg)\epsilon t^2,
$$
from which (2) follows directly for an appropriate choice of $C_2=C_2(r)$.  For (3), parts (1) and (2) yield 
$$
|\{ij\in E(R) : f(i,j)= m(r)\}| = {t\choose 2}-a-b \geq  {t\choose 2} - (C_1+C_2)\epsilon t^2 = \Bigg(\frac{1}{2} -(C_1+C_2)\epsilon\Bigg)t^2 - \frac{t}{2}.
$$
Setting $m=|\{ u\in V(R): \mu^R_{m(r)}(u)<(1-\sqrt{\epsilon})(t-1)\}|$, it is clear that
$$
\sum_{v\in V(R)} \mu^R_{m(r)}(v) \leq m(1-\sqrt{\epsilon})(t-1) + (t-m)(t-1) = t^2-t-\sqrt{\epsilon}mt+\sqrt{\epsilon}m.
$$
On the other hand,  let $\mathcal{G}$ be the graph with vertex set $\mathcal{V}=[t]$ and edge set $\mathcal{E}=\{ij\in {\mathcal{V} \choose 2}: f^R(ij)=m(r)\}$.  Then 
$$
\sum_{v\in V(R)} \mu^R_{m(r)}(v)=\sum_{v\in \mathcal{V}} \mathcal{DEG}(v)=2|\mathcal{E}| \geq 2 \Bigg(\Bigg(\frac{1}{2}-\epsilon (C_1+C_2)\Bigg)t^2 -\frac{t}{2}\Bigg)=  (1-2\epsilon (C_1+C_2))t^2 -t.
$$
Consequently $(1-2\epsilon (C_1+C_2))t^2 -t\leq t^2-t-\sqrt{\epsilon}mt+\sqrt{\epsilon}m$.  Simplifying this we obtain 
$$
m\leq \frac{2\epsilon (C_1+C_2)t^2}{\sqrt{\epsilon}(t-1)} = 2\sqrt{\epsilon}(C_1+C_2) \frac{t^2}{t-1}.
$$
Set $C_3=3(C_1+C_2)$.  It is now clear that there is $M_2$ such that if $t>M_2$, then $m\leq \sqrt{\epsilon}C_3 t$, so (3) holds.  Therefore if $t> M=\max\{ M_1, M_2\}$, (1), (2), and (3) hold.
\end{proof}

\medskip

%******************************************************************************************************************************
\subsection{Proof of Theorem \ref{stabthm}}
%******************************************************************************************************************************
In this section we will prove our stability result below.

\begin{theorem}\label{stabthm}
Fix an integer $r\geq 3$.  For all $\delta>0$ there is $0<\epsilon<1$ and $M$ such that for all $t>M$ the following holds. If $R\in \tilde{M}_r(t)$ and $W(R)>m(r)^{(1-\epsilon){t\choose 2}}$, then $R$ is $\delta$-close to $\tilde{C}_r(t)$.
\end{theorem}

The following is a consequence of Lemma \ref{sizelemma}, so its proof appears in the appendix along with the proof of Lemma \ref{sizelemma}.

\begin{lemma}\label{lemmatriangle}
Suppose $r\geq 3$ is an integer and $A,B,C\subseteq [r]$ are such that $|A|=|B|=|C|=m(r)$ and there is no violating triple $(a,b,c)\in A\times B\times C$.  Then one of the following holds:
\begin{enumerate}
\item $r$ is even and $A=B=C=[m(r)-1,r]$.
\item $r$ is odd and for some relabeling $\{A,B,C\}=\{D,E,F\}$ one of the following holds:
\begin{enumerate}
\item $D=F=E=[m(r)-1, r-1]$.
\item $D=F=[m(r),r]$, $E\subseteq [m(r)-1,\ldots r]$.
\end{enumerate}
\end{enumerate}
\end{lemma}

An immediate corollary of this is the following.

\begin{corollary}\label{keycor}
Suppose $r, t\geq 3$ are integers, $R\in \tilde{M}_r(t)$, and $xy,yz,xz\in {[t]\choose 2}$ are such that $f^R(x,y)=f^R(y,z)=f^R(x,z)=m(r)$.  Then one of the following holds:
\begin{enumerate}
\item $r$ is even and $c^R(xy)=c^R(yz)=c^R(xz) = [m(r)-1,r]$.
\item $r$ is odd and for some relabeling $\{x,y,z\}=\{u,v,z\}$ one of the following holds:
\begin{enumerate}
\item $c^R(uv)=c^R(uw)=c^R(vw)=[m(r)-1,r-1]$.
\item $c^R(uv)=c^R(uw)=[m(r),r]$, $c^R(v,w)\subseteq [m(r)-1, r]$.
\end{enumerate}
\end{enumerate}
\end{corollary}

\begin{proof}
$R\in \tilde{M}_r(t)$ implies there is no violating triple $(a,b,c)\in c^R(uv)\times c^R(uw)\times c^R(vw)$.  Thus the corollary follows immediately by applying Lemma \ref{lemmatriangle} to $A=c^R(uv)$, $B=c^R(uw)$ and $C=c^R(vw)$.
\end{proof}

We will use the following consequence of Corollary \ref{keycor}.

\begin{lemma}\label{keylem}
For all integers $r\geq 3$ and $0<\epsilon<1$, there is $M$ such that $t>M$ and $R\in \tilde{M}_r(t)$ implies the following.  Let $V=[t]$ and $V_0=\{u\in V: \mu^R_{m(r)}(u)< (1-\sqrt{\epsilon})(t-1)\}$.
\begin{enumerate}[(i)]
\item  $r$ is even and for all $xy\in {V\choose 2}\setminus E(V,V_0)$, $f^R(x,y) =m(r)$ implies $c^R(xy)=[m(r)-1,r]$.
\item $r$ is odd and for all $xy\in {V\choose 2}\setminus E(V,V_0)$, $f^R(x,y)=m(r)$ implies one of the following:
\begin{enumerate}
\item either $r\in c^R(xy)$ or $c^R(xy)= [m(r)-1, r-1]$.
\item either $m(r)-1\in c^R(xy)$ or $c^R(xy)= [m(r), r]$.
\end{enumerate} 
\end{enumerate} 
\end{lemma}
\begin{proof}
Fix an integer $r\geq 3$ and $0<\epsilon <1$.  Choose $M$ large enough so that $t>M$ implies $t-2-2\sqrt{\epsilon}(t-1)\geq 1$ and fix $R\in \tilde{M}_r(t)$.  Suppose $xy\in {V\choose 2}\setminus E(V_0,V)$ and $f^R(x,y)=m(r)$.  Since $x,y\notin V_0$, $\min \{ \mu^R_{m(r)}(x), \mu^R_{m(r)}(y)\} \geq(1-\sqrt{\epsilon})(t-1)$.  Therefore
$$
|(V\setminus\{x,y\}) \cap \Gamma^R_{m(r)}(x)\cap \Gamma^R_{m(r)}(y)| \geq t- 2-2\sqrt{\epsilon}(t-1) \geq 1,
$$
where the last inequality holds by our assumption on $M$.  Thus there is $z\in V\setminus \{x,y\}$ such that $f^R(x,y)=f^R(y,z)=f^R(x,z)=m(r)$.  If $r$ is even, part (1) of Corollary \ref{keycor} implies $c^R(xy)=[m(r)-1,r]$, so (i) holds.  If $r$ is odd, part (2) of Corollary \ref{keycor} implies $c^R(xy)\subseteq [m(r)-1,r]$.  Recall that since $r$ is odd, $|[m(r)-1,r]|=m(r)+1$.  Therefore, since $|c^R(xy)|=f^R(x,y)=m(r)$ and $c^R(xy)\subseteq [m(r)-1,r]$, $m(r)-1\notin c^R(xy)$ implies $c^R(xy)= [m(r),r]$, and $r\notin c^R(xy)$ implies $c^R(xy)=[m(r)-1,r-1]$, so (ii) holds.
\end{proof}

\noindent{\bf Proof of Theorem \ref{stabthm}}. Let $r\geq3$ be an integer, and fix $\delta>0$.  Let $C_1$, $C_2$, $C_3$ be as in Lemma \ref{boundnonf(ij)=m(r)}. We will consider the cases when $r$ is even and odd separately.

\underline{Case $1$}: $r$ is even.  Fix $0<\epsilon<1$ small enough so that $\max\{\sqrt{\epsilon}C_3, (C_1+C_2)\epsilon \}<\frac{\delta}{2}$. Apply Lemma \ref{boundnonf(ij)=m(r)} to $\epsilon$ to obtain $M_1$, and apply Lemma \ref{keylem} to $\epsilon$ to obtain $M_2$.  Fix  $M>\max \{ M_1,M_2\}$ large enough so that $t>M$ implies $\frac{t^2}{t-1}\leq 2t$.  Fix $t>M$ and $R\in \tilde{M}_r(t)$ such that $W(R)\geq m(r)^{(1-\epsilon){t\choose 2}}$.  Set $V=[t]$.  Let $R'$ be the unique element of $\tilde{C}_r(t)$, that is, $R'$ is the complete $r$-graph with vertex set $V$ such that for all $xy \in {V\choose 2}$, $c^{R'}(xy)=[m(r)-1,r]$.  We show $|\Delta(R,R')|\leq \delta t^2$.

Let $V_0=\{u\in V: \mu^R_{m(r)}(u)<(1-\sqrt{\epsilon})(t-1)\}$ and $V_1 = V\setminus V_0$.  Note that this is the same definition of $V_0$ used in Lemma \ref{keylem}.  Define $A = E(V_0,V) \cup \{xy \in {V\choose 2} : f^R(x,y)\neq m(r)\}$.  Suppose $xy \in {V\choose 2}\setminus A$.  Then $xy\in {V\choose 2}\setminus E(V,V_0)$ and $f^R(x,y)=m(r)$, so Lemma \ref{keylem} (i) implies $c^R(xy)=[m(r)-1,r]$.  Thus $c^R(xy)=c^{R'}(xy)$ and $xy \notin \Delta(R,R')$.  We have shown $\Delta(R,R')\subseteq A$, and consequently $|\Delta(R',R)|\leq |A|$.  

We now bound $|A|$.  The definition of $A$ and parts (1), (2), and (3) of Lemma \ref{boundnonf(ij)=m(r)} imply
$$
|A|\leq |V||V_0|+a_R +b_R \leq (\sqrt{\epsilon}C_3 + (C_1+C_2)\epsilon )t^2.
$$
By assumption on $\epsilon$, $(\sqrt{\epsilon}C_3 + (C_1+C_2)\epsilon )t^2<(\frac{\delta}{2}+\frac{\delta}{2})t^2 = \delta t^2$, and consequently, $|\Delta(R,R')|\leq \delta t^2$ as desired.

\underline{Case 2}: $r$ is odd.  Fix $0<\epsilon<1$ small enough so that $\max \{\sqrt{\epsilon}C_3, (C_1+C_2)\epsilon, 2\sqrt{\epsilon}\}<\frac{\delta}{5}$.  Apply Lemma \ref{boundnonf(ij)=m(r)} to $\epsilon$ to obtain $M_1$ and apply Lemma \ref{keylem} to $\epsilon$ to obtain $M_2$.  Choose $M>\max\{M_1,M_2\}$ large enough so that $t>M$ implies $\frac{2}{\sqrt{\epsilon}t}<\frac{\delta}{5}$, $\sqrt{\epsilon}t^2+t\leq 2\sqrt{\epsilon}t^2$, and $ \frac{t^2}{t-1}<2t$.  Fix $t>M$ and $R\in \tilde{M}_r(t)$ such that $W(R)\geq m(r)^{(1-\epsilon){t\choose 2}}$ and set $V=[t]$.  We construct an element $R'\in \tilde{C}_r(t)$,  then show $|\Delta(R,R')|\leq \delta t^2$.  First we choose a partition $V_0, V_1,\ldots, V_l, \ldots, V_k$ of $V$ with the following properties: 
\begin{enumerate}[$\bullet$]
\item $|V_0|\leq \sqrt{\epsilon}C_3 t$,
\item If $0<l$, then for each $1\leq i\leq l$, there is $u_i\in V$ and $B_i\subseteq V$ such that  $V_i=(N^R_{m(r)-1}(u_{i}) \cap B_i)\cup \{u_{i}\}$,
\item If $l<k$, then $V_{l+1},\ldots, V_k$ are singletons.
\end{enumerate}
\underline{Step $1$}: Let $V_0=\{u\in V: \mu^R_{m(r)}(u)<(1-\sqrt{\epsilon})(t-1)\}$.  Note that part (3) of Lemma \ref{boundnonf(ij)=m(r)} implies $|V_0|\leq \sqrt{\epsilon}C_3 t$.  Define $B_1 = V\setminus V_0$.  If there exists $u\in B_1$ such that $|N^{R}_{m(r)-1}(u)\cap B_1|\geq \sqrt{\epsilon}(t-1)$, then choose $u_{1}$ to be any $u\in B_1$ with $|N^{R}_{m(r)-1}(u)\cap B_1|$ maximal, and set $V_{1} = (N^R_{m(r)-1}(u_{1}) \cap B_1)\cup \{u_{1}\}$.  If $V\setminus (V_0\cup V_1)=\emptyset$, set $k=l=1$ and end the construction.  If not, go to step 2.  

If no $u$ exists in $B_1$ such that $|N^{R}_{m(r)-1}(u)\cap B_1|\geq \sqrt{\epsilon}(t-1)$, then put each element of $B_1$ into its own part and end the construction.  In particular, set $l=0$, $k=t-|V_0|$, and let $V_2,\ldots, V_k$ partition $B_1$ into singletons.
\underline{Step $i+1$}: Suppose $i\geq 1$ and we have chosen $V_i, B_i$, and $u_i$ such that $V_i=(N^R_{m(r)-1}(u_{i}) \cap B_i)\cup \{u_{i}\}$ and $V\setminus \bigcup_{j=0}^iV_j \neq \emptyset$. Set $B_{i+1}=V\setminus \bigcup_{j=0}^i V_j$.  If there exists $u\in B_{i+1}$ such that $|N^R_{m(r)-1}(u) \cap B_{i+1}|\geq \sqrt{\epsilon}(t-1)$, choose $u_{i+1}$ to be any $u\in B_{i+1}$ with $|N^R_{m(r)-1}(u) \cap B_{i+1}|$ maximal, and set $V_{i+1} = (N^R_{m(r)-1}(u_{i+1}) \cap B_{i+1})\cup \{u_{i+1}\}$. If $V\setminus \bigcup_{j=0}^{i+1} V_j = \emptyset$, set $k=l=i+1$ and end the construction.   Otherwise go to step $i+2$.

If no $u$ exists in $B_{i+1}$ such that $|N^R_{m(r)-1}(u) \cap B_{i+1}|\geq \sqrt{\epsilon}(t-1)$, then put each element of $B_{i+1}$ into its own part and end the construction.  In particular, set $l=i$, $k=t-|\bigcup_{j=0}^i V_i|$, and let $V_{i+1},\ldots, V_k$ partition $B_{i+1}$ into singletons.

This completes the construction of the partition $V_0, V_1,\ldots, V_l,\ldots, V_k$.  Given $xy\in {V\choose 2}$, define
\[ c^{R'}(xy) = \begin{cases}
[m(r)-1,r-1] & \textnormal{ if } xy\in {V_i\choose 2} \textnormal{ some }0\leq i\leq l \\
[m(r),r] & \textnormal{ otherwise}.
\end{cases}
\]
This completes our construction of $R'$.  We now bound $|\Delta(R,R')|$.  Set 
$$
A = E(V_0,V)\cup \Bigg\{xy\in {V\choose 2}: f^R(x,y)\neq m(r)\Bigg\}\cup \bigcup_{i=1}^l E(\{u_i\},V) \cup E(V_i, V\setminus \Gamma^R_{m(r)}(u_i)).
$$
We first bound $|A|$, then $|\Delta(R,R')\setminus A|$.  By parts (1), (2), and (3) of Lemma \ref{boundnonf(ij)=m(r)}, 
$$
\Bigg|E(V_0,V)\cup \Bigg\{xy\in {V\choose 2}: f^R(x,y)\neq m(r)\Bigg\} \Bigg|\leq |V||V_0|+a_R+b_R \leq C_3\sqrt{\epsilon}t^2+ C_1\epsilon t^2+C_2\epsilon t^2.
$$
By construction, for each $1\leq i\leq l$, $|V_i|\geq \sqrt{\epsilon}(t-1)$, therefore $l\leq \frac{t}{\sqrt{\epsilon}(t-1)}$.  Thus 
$$
\Bigg|\bigcup_{i=1}^l E(\{u_i\},V)\Bigg|\leq lt \leq \frac{t^2}{\sqrt{\epsilon}(t-1)} \leq \frac{2t}{\sqrt{\epsilon}},
$$
where the last inequality is by assumption on $M$.  By construction, for each $1\leq i\leq l$, $u_i\notin V_0$ implies $|V\setminus \Gamma^R_{m(r)}(u_i)|\leq \sqrt{\epsilon}(t-1)+1$. Therefore 
$$
\Bigg|\bigcup_{i=1}^l E(V_i,V\setminus \Gamma^R_{m(r)}(u_i))\Bigg|\leq \sum_{i=1}^l |V_i||V\setminus \Gamma^R_{m(r)}(u_i)| \leq (\sqrt{\epsilon}(t-1)+1)\sum_{i=1}^l |V_i| \leq (\sqrt{\epsilon}(t-1)+1) t\leq 2\sqrt{\epsilon}t^2,
$$
where the last inequality is by assumption on $M$.  Combining all of this yields that 
$$
|A|\leq \Bigg(\sqrt{\epsilon}C_3+ (C_1+C_2)\epsilon + \frac{2}{\sqrt{\epsilon}t}+2\sqrt{\epsilon}\Bigg) t^2.
$$
We now bound $|\Delta(R,R')\setminus A|$. An edge $xy\in \Delta(R,R')\setminus A$ is contained in one of the following:
\begin{enumerate}[$\bullet$]
\item $X=\bigcup_{l+1\leq i<j\leq k} \{xy\in E(V_i,V_j)\setminus A: c^R(x,y)\neq [m(r), r]\}$.
\item For some $1\leq i\leq l$, $Y_{i}=\{xy \in E(V_i) \setminus A :c^R(x,y)\neq [m(r)-1,r-1]\}$.
\item For some $1\leq i< j\leq l$, $Z_{ij}=\{xy \in E(V_i,V_j)\setminus A : c^R(x,y)\neq [m(r),r]\}$.
\item For some $1\leq i\leq l<j\leq k$, $W_{ij}=\{ xy\in E(V_i,V_j)\setminus A: c^R(x,y)\neq [m(r),r]\}$.
\end{enumerate}
We now bound $|X|$.  Define $\mathcal{G}$ to be the graph with vertex set $\mathcal{V}= \bigcup_{j=l+1}^k V_j$ and edge set 
$$
\mathcal{E}=\{xy\in {\mathcal{V}\choose 2}: m(r)-1\in c^R(xy)\}.
$$
By definition of $X$, for all $xy\in X$ we have $xy\in {V\choose 2}\setminus E(V_0,V)$, $f^R(x,y) = m(r)$, and  $c^R(xy)\neq [m(r),r]$, so Lemma \ref{keylem} (ii)(b) implies  $m(r)-1\in c^R(xy)$, and therefore $X\subseteq \mathcal{E}$.  By construction, for all $u\in \mathcal{V}$, $\mathcal{DEG}(v) = |N^{R}_{m(r)-1}(u) \cap \mathcal{V}| < \sqrt{\epsilon}(t-1)$, thus
$$
|X|\leq |\mathcal{E}|=\frac{1}{2}\sum_{v\in V}\mathcal{DEG}(v) < \frac{\sqrt{\epsilon}t^2}{2}.
$$
We now show each $Y_i$ is empty. If $l=0$ this is vacuous, so assume $l\geq 1$.  Suppose for a contradiction that for some $1\leq i\leq l$, $Y_i\neq \emptyset$.  Then there is $xy\in E(V_i)$ such that $f^R(x,y)=f^R(x,u_i)=f^R(y,u_i)=m(r)$ and $c^R(xy) \neq [m(r)-1,r-1]$.  By Lemma \ref{keylem} (ii)(a), $r\in c^R(xy)$.  But by construction, $m(r)-1\in c^R(u_ix)\cap c^R(u_iy)$.  Now $(r,m(r)-1,m(r)-1)\in c^R(xy)\times c^R(u_ix)\times c^R(u_iy)$ is a violating triple, making $\{x,y,u_i\}$ a violating triangle, a contradiction.  

We now show each $Z_{ij}$ is empty.  If $l<2$ this is trivial, so assume $l\geq 2$.  Suppose for a contradiction that for some $1\leq i<j\leq l$, there is $xy \in Z_{ij}$, say with $x\in V_i$, $y\in V_j$.  Then $f^R(x,y)=f^R(u_i,y)=f^R(u_i,x)=m(r)$ and $c^R(xy)\neq [m(r), r]$.  By Lemma \ref{keylem} (ii)(b), $m(r)-1 \in c^R(xy)$, and by construction $m(r)-1\in c^R(xu_i)$.  Also by construction, $m(r)-1\notin c^R(u_iy)$, so Lemma \ref{keylem} (ii)(b) implies $c^R(u_iy)=[m(r), r]$. But now $(r,m(r)-1,m(r)-1)\in c^R(u_iy)\times c^R(u_ix)\times c^R(xy)$ is a violating triple, making $\{u_i,x,y\}$ a violating triangle, a contradiction.

We now show each $W_{ij}=\emptyset$.  If $l=0$ or $k=l$, this is vacuous, so assume $1\leq l<k$.  Fix $1\leq i\leq l$ and $l+1\leq j\leq k$ and suppose for a contradiction there is $xy\in W_{ij}$, say with $x\in V_i$, $y\in V_j$.  Then $f^R(x,y)=f^R(u_i,y)=f^R(u_i,x)=m(r)$ and $c^R(xy)\neq [m(r),r]$.  By Lemma \ref{keylem} (ii)(b) $m(r)-1\in c^R(xy)$, and by construction $m(r)-1\in c^R(xu_i)$. Also by construction, $m(r)-1\notin c^R(u_iy)$, so Lemma \ref{keylem} (ii)(b) implies that $c^R(u_iy) =[m(r),r]$.  But now $(r,m(r)-1,m(r)-1)\in c^R(u_iy)\times c^R(u_ix)\times c^R(xy)$ is a violating triple, making $\{u_i,x,y\}$ a violating triangle, a contradiction.

Combining all of this yields that $|\Delta(R,R')\setminus A|\leq  \frac{\sqrt{\epsilon}t^2}{2}$, so 
$$
|\Delta(R,R')|\leq (\sqrt{\epsilon}C_3+ (C_1+C_2)\epsilon + \frac{2}{\sqrt{\epsilon}t}+2\sqrt{\epsilon} + \frac{\sqrt{\epsilon}}{2})t^2.
$$
By our assumptions on $\epsilon$ and because $\frac{2}{\sqrt{\epsilon}t}<\frac{\delta}{5}$, $(\sqrt{\epsilon}C_3+ (C_1+C_2)\epsilon + \frac{2}{\sqrt{\epsilon}t}+2\sqrt{\epsilon} + \frac{\sqrt{\epsilon}}{2})t^2< 5\frac{\delta}{5} = \delta$, and $|\Delta(R,R')| <\delta t^2$ as desired.  
\qed

%************************************************************************
\section{Proof of Theorem \ref{deltaclosethm}}\label{provingtheorem3}

%************************************************************************
\setcounter{theorem}{0}
\numberwithin{theorem}{section}

In this section we prove Theorem \ref{deltaclosethm}, which says that for all integers $r\geq 3$ and all $\delta>0$, almost all elements of $M_r(n)$ are $\delta$-close to $C_r(n)$.  We begin with some key definitions.  For $n,r,s\geq 3$ integers, and $\delta, \eta, d, \epsilon \geq 0$, set
\begin{align*}
\tilde{C}^{\delta}_r(t) &= \{R\in \tilde{M}_r(t): R\textnormal{ is }\delta\textnormal{-close to }\tilde{C}_r(t)\}, \\
D_r(n, \delta, \eta, d) &= \{G\in M_r(n): Q_{\eta,d}(G)\neq \emptyset \textnormal{ and for all }R\in Q_{\eta, d}(G), R\in \tilde{C}^{\delta}_r(t)\textnormal{ where }t=|V(R)| \},\\
\tilde{E}_r(s, \epsilon) &= \{R\in \tilde{M}_r(s): W(R)\geq m(r)^{(1-\epsilon){s\choose 2}}\},\textnormal{ and }\\
E_r(n,\epsilon, \eta, d) &=\{G\in M_r(n): \textnormal{ for all }R\in Q_{\eta,d}(G), R\in \tilde{E}_r(t, \epsilon)\textnormal{ where } t=|V(R)|\}, 
\end{align*}
and recall that $C_r^{\delta}(n) =\{G\in M_r(n): G\textnormal{ is }\delta \textnormal{-close to }C_r(n)\}$.  Theorem \ref{deltaclosethm} follows from two lemmas that we now prove.  The first lemma below informally states that $r$-graphs in $M_r(n)$ with reduced $r$-graphs close to $\tilde{C}_r(t)$ are themselves close to $C_r(n)$. 

\begin{lemma}\label{lemma5}
Let $r,n\geq 3$ be integers.  For all $\delta>0$, there is $d_0$ such that for all $d\leq d_0$ and $\eta\leq \gamma_{el}(d,3)$, 
$$
D_r(n, \delta/2, \eta, d)\subseteq C^{\delta}_r(n).
$$

\end{lemma}

\begin{proof}
Fix $\delta>0$ and set $d_0=\frac{\delta}{2(r+2)}$.  Fix $d\leq d_0$ and $\eta \leq \gamma_{el}(d,3)$, and suppose $G\in D_r(n, \delta/2, \eta, d)$.   Then by definition of $D_r(n, \delta/2, \eta, d)$, $G\in M_r(n)$ and there is $R \in Q_{\eta, d}(G)$ which is $\frac{\delta}{2}$-close to $\tilde{C}_r(t)$ where $t=|V(R)|$.  Let $R'\in \tilde{C}_r(t)$ be such that $R$ is $\frac{\delta}{2}$-close to $R'$.  We will build an element $G'\in C_r(n)$ such that $G$ is $\delta$-close to $G'$. 

Let $\mathcal{P}=\{V_1,\ldots, V_t\}$ be an $\eta$-regular partition for $G$ such that $R=R(G,\mathcal{P},d)$.   Define 
$$
A=\Delta(R,R') \cup \{ij\in {[t]\choose 2}: (V_i, V_j)\textnormal{ is not }\eta\textnormal{-regular for }G\}.
$$
Note that $|A| \leq \frac{\delta}{2} t^2 +\eta t^2$.  
Define $G'$ by $V(G')=V(G)=[n]$ and for $xy\in {[n]\choose 2}$, 
\[
d^{G'}(x,y)= \begin{cases}
r-1 & \textnormal{ if }xy\in E(V_i)\textnormal{ for some }i\in [t]\\
r-1 & \textnormal{ if }xy\in E(V_i,V_j) \textnormal{ for some }ij\in {[t]\choose 2}\textnormal{ such that either }ij\in A\textnormal{ or }d^G(x,y) \notin c^{R'}(ij)\\
d^G(x,y)& \textnormal{ if }xy\in E(V_i,V_j) \textnormal{ for some } ij\in {[t]\choose 2}\setminus A \textnormal{ and }d^G(x,y) \in c^{R'}(ij).
\end{cases}
\]
Set 
$$
\begin{array}{ccc}
U_r = \begin{cases}
[m(r),r] & \textnormal{ if }r\textnormal{ is odd}\\
[m(r)-1,r] & \textnormal{ if }r\textnormal{ is even}
\end{cases}
 &\quad \text{and} \quad &
L_r = \begin{cases}
[m(r)-1,r-1] & \textnormal{ if }r\textnormal{ is odd}\\
[m(r)-1,r] & \textnormal{ if }r\textnormal{ is even}.
\end{cases}
\end{array}
$$
Note that $r-1\in U_r\cap L_r$.  By definition of $\tilde{C}_r(t)$, there is a partition $\tilde{W}_1,\ldots, \tilde{W}_s$ of $[t]$ such that for all $ij\in {[t]\choose 2}$,
\[
c^{R'}(ij)=\begin{cases}
L_r & \textnormal{ if } ij\in E(\tilde{W}_u) \textnormal{ some }u\in [s]\\
U_r & \textnormal{ if } ij\in E(\tilde{W_u},\tilde{W}_v) \textnormal{ some }uv\in {[s]\choose 2}.
\end{cases}
\]
Define a new partition $W_1,\ldots, W_s$ of $[n]$ by setting $W_u = \bigcup_{i\in W_u} V_i$ for each $u\in [s]$.  Then by construction, for all $xy\in {[n]\choose 2}$, 
\[
d^{G'}(x,y)\in \begin{cases}
L_r & \textnormal{ if } xy\in E(W_u) \textnormal{ some }u\in [s]\\
U_r & \textnormal{ if } xy\in E(W_u,W_v) \textnormal{ some }uv\in {[s]\choose 2}.
\end{cases}
\]
Therefore, $G'\in C_r(n)$ by definition. We now show $|\Delta(G,G')|\leq \delta n^2$.  Recall that by definition of $Q_{\eta,d}(G)$, $\frac{1}{\eta}\leq t$.  Edges $xy\in \Delta(G,G')$ fall into the following categories:
\begin{enumerate}[$\bullet$]
\item $xy\in E(V_i)$ some $i\in [t]$.  There are at most $t (\frac{n^2}{2t^2}) = \frac{n^2}{2t}< \eta n^2$ such edges.
\item $xy \in E(V_i,V_j)$ some $ij\in A$.  There are at most $|A|\frac{n^2}{t^2}\leq (\frac{\delta}{2} t^2+\eta t^2)\frac{n^2}{t^2} = (\frac{\delta}{2}+\eta)n^2$ such edges.
\item $xy \in E(V_i,V_j)$ some $ij\in {[t]\choose 2} \setminus A$ such that $d^G(x,y)\notin c^{R'}(ij)$.  This means $(V_i,V_j)$ is $\eta$-regular for $G$ and $c^{R'}(ij)=c^R(ij)$.  Because $R=R(G,\mathcal{P},d)$, for each $l\in [r]\setminus c^R(ij)$ we have $e^G_l(V_i,V_j) \leq d|V_i||V_j|$.  Therefore the number of such edges is at most $d\frac{n^2}{t^2}r{t\choose 2}< drn^2$.
\end{enumerate}

Combining these bounds with the fact that $\eta\leq d\leq d_0 = \frac{\delta}{2(r+2)}$ yields
$$
|\Delta(G,G')|\leq n^2(2\eta + \frac{\delta}{2}+ dr) \leq n^2(2d_0+\frac{\delta}{2} + d_0r) =n^2(\frac{\delta}{2}+d_0(r+2)) = \delta n^2.
$$
\end{proof}

We now prove the second lemma.  Informally, it says that most graphs in $M_r(n)$ have all their reduced graphs $R$ with $W(R)$ quite large.
\begin{lemma}\label{lemma3}
For all $\epsilon>0$, there is $\beta=\beta(\epsilon)$ and $d_0=d_0(\epsilon)>0$, such that for all $d\leq d_0$ and $\eta\leq \gamma_{el}(d,3)$, there is $M$ such that $n\geq M$ implies
\begin{eqnarray}\label{explicit*}
\frac{|M_r(n)\setminus E_r(n,\epsilon,\eta, d)|}{m(r)^{n\choose 2}} \leq 2^{-\beta n^2}.
\end{eqnarray}
\end{lemma}

\begin{proof}
 All logs in this proof are base $2$.  Fix $\epsilon>0$ and set $\beta=\frac{\epsilon \log m(r)}{8}$.  Define 
$$
F(x) = \frac{3x}{2}\log r +r(H(x)+x) -2\beta,
$$ and choose $d_0<\frac{1}{2}$ small enough so that $F(d_0) < -\beta$.  Recall that for $0\leq y\leq x\leq \frac{1}{2}$, $H(y)\leq H(x)$, so for any $0\leq y\leq x\leq d_0$, $F(y)\leq F(x)$.  Fix $d\leq d_0$ and $\eta\leq \gamma_{el}(d,3)\leq d$.  Set $N=CM(\frac{1}{\eta},\eta)$ and define 
\begin{enumerate}[ ]
\item $C=\log(N-\frac{1}{\eta}+1) + \log r {N\choose 2} + (H(\eta)+\eta)N^2$,
\item $C' = \log N + \frac{\log m(r)}{2}$, and 
\item $C''= \frac{3\eta}{2}\log r + r(H(d)+d) - 4\beta $.   
\end{enumerate}
For any integer $n$, define
\begin{enumerate}[ ]
\item $S(n)=n\log(N)+(H(\eta) +\eta) N^2 + (\frac{3\eta}{2}\log r) n^2 + r(H(d)+d)n^2$,
\item $S'(n)= S(n)+\log(N- \frac{1}{\eta}+1)+{N\choose 2}\log r$, and 
\item $S''(n)= S'(n)-4\beta n^2+\frac{\log m(r)}{2}n$.   
\end{enumerate}
Notice that $S''(n)=C+C'n+C''n^2$ and $C''=F(\eta) -2\beta$.  Choose $M\geq N$ large enough so that $n>M$ implies $S''(n) < (C''+2\beta) n^2=F(\eta)n^2$.  We show $n>M$ implies (\ref{explicit*}) holds.  Fix $n>M$.  Our assumptions on $\eta \leq d_0$ and $M$ imply $S''(n)<F(\eta)n^2\leq F(d_0)n^2<-\beta n^2$, so it suffices to show
\begin{eqnarray}\label{**}
\frac{|M_r(n)\setminus E_r(n, \epsilon, \eta, d)|}{m(r)^{n\choose 2}} \leq 2^{S''(n)}.
\end{eqnarray}
By definition of $E(n,\epsilon, \eta, d)$, we have $G\in M_r(n)\setminus E_r(n, \epsilon, \eta, d)$ if and only if there is $\frac{1}{\eta}\leq t\leq N$ and $R\in \tilde{M}_r(t)$ such that $R\in Q_{\eta,d}(G)$ and $W(R)<m(r)^{(1-\epsilon){t\choose 2}}$.  We give an upper bound for the number of such $G$.  

Fix some $\frac{1}{\eta}\leq t\leq N$ and $R\in \tilde{M}_r(t)$ such that $W(R)<m(r)^{(1-\epsilon){t\choose 2}}$.  All $G\in M_r(n)$ such that $R\in Q_{\eta,d}(G)$ can be constructed as follows:
\begin{enumerate}[$\bullet$]
\item Choose an equipartition of $[n]$ into $t$ pieces $V_1,\ldots, V_t$. There are at most $t^n\leq N^n$ such partitions.  Note that for each $i\in [t]$, $|V_i|\leq \eta n$.
\item Choose $J\subseteq {[t]\choose 2}$ to be the set of $ij$ such that $(V_i,V_j)$ is not $\eta$-regular for $G$. There are at most ${{t\choose 2}\choose \eta {t\choose 2}} 2^{\eta {t\choose 2}} \leq 2^{H(\eta)t^2+\eta t^2}\leq 2^{(H(\eta)+\eta)N^2}$ ways to do this.
\item Choose $d^G(x,y)$ for each $xy\in E(V_i)$ and $i\in [t]$.  There are at most $r^{t(\frac{n^2}{2t^2})}=r^{\frac{n^2}{2t}}\leq  r^{\frac{\eta}{2} n^2}$ ways to do this.
\item Choose $d^G(x,y)$ for each $xy\in E(V_i,V_j)$ where $ij\in J$. There are at most $(r^{\frac{n^2}{t^2}})^{\eta t^2} =r^{\eta n^2}$ ways to do this.
\item Choose $d^G(x,y)$ for each $xy\in E(V_i,V_j)$ where $ij\in I={[t]\choose 2}\setminus J$.  For each $ij\in I$, $(V_i, V_j)$ is $\eta$-regular, so the colors for edges in $E(V_i,V_j)$ can be chosen as follows:
\begin{enumerate}
\item For each $s\notin c^R(ij)$, choose a subset of $E(V_i,V_j)$ of size at most $d|V_i||V_j|$ to have color $s$.  The number of ways to do this is at most $({\frac{n^2}{t^2} \choose d  \frac{n^2}{t^2}} 2^{d \frac{n^2}{t^2}})^r \leq 2^{r(H(d)\frac{n^2}{t^2}+d \frac{n^2}{t^2})}$.
\item Assign colors from $c^R(ij)$ to the rest of the edges in $E(V_i,V_j)$.  There are at most $f^R(i,j)^{\frac{n^2}{t^2}}$ ways to do this.
\end{enumerate}
Therefore, the total number of ways to choose $d^G(x,y)$ for $xy\in E(V_i,V_j)$ where $ij\in I$ is at most
\begin{align*}
 & \prod_{ij\in I} 2^{r(H(d)+d)\frac{n^2}{t^2}} f^R(i,j)^{\frac{n^2}{t^2}}
\leq  2^{r(H(d)+d)n^2}\Bigg( \prod_{ij\in I} f^R(i,j)^{\frac{n^2}{t^2}}\Bigg)  \leq  2^{r(H(d)+d)n^2} W(R)^{\frac{n^2}{t^2}}.
\end{align*}
\end{enumerate}
By assumption, $W(R)< m(r)^{(1-\epsilon){t\choose 2}}$.  Therefore 
$$
W(R)^{\frac{n^2}{t^2}}< m(r)^{(1-\epsilon){t\choose 2}\frac{n^2}{t^2}} <m(r)^{(1-\epsilon)({n\choose 2} +\frac{n}{2}) }.
$$
Combining the above yields that the number of $G\in M_r(n)$ with $R\in Q_{\eta,d}(G)$ is at most
\begin{align*}
N^n 2^{(H(\eta)+\eta)N^2}r^{\frac{ \eta}{2} n^2}r^{\eta n^2}  2^{r(H(d)+d)n^2} m(r)^{(1-\epsilon)({n\choose 2} +\frac{n}{2})} = 2^{S(n)} m(r)^{(1-\epsilon)({n\choose 2} +\frac{n}{2})}.
\end{align*}
The number of $R\in \tilde{M}_r(t)$ with $\frac{1}{\eta}\leq t\leq N$ is at most $(N-\frac{1}{\eta}+1)|\tilde{M}_r(N)|$, so
\begin{align*}\nonumber
 |M_r(n)\setminus E_r(n, \epsilon, \eta, d)|&<(N-\frac{1}{\eta}+1) |\tilde{M}_r(N)| 2^{S(n)} m(r)^{(1-\epsilon)({n\choose 2} +\frac{n}{2})} \\
 &<  (N-\frac{1}{\eta}+1)r^{{N \choose 2}} 2^{S(n)} m(r)^{(1-\epsilon)({n\choose 2} +\frac{n}{2})} = 2^{S'(n)}m(r)^{(1-\epsilon)({n\choose 2} +\frac{n}{2})}.
\end{align*}
Thus
$$
\frac{|M_r(n)\setminus E_r(n, \epsilon, \eta, d)|}{m(r)^{n\choose 2}} < \frac{2^{S'(n)}m(r)^{(1-\epsilon)({n\choose 2} +\frac{n}{2} )}}{m(r)^{{n\choose 2}}} =2^{S''(n)}.
$$ 
We have shown that $n>M$ implies (\ref{**}) holds, so we are done.
\end{proof}

\medskip

\noindent{\bf Proof of Theorem \ref{deltaclosethm}}.  Fix $\delta >0$.  Apply Theorem \ref{stabthm} to $\frac{\delta}{2}$ to obtain $\epsilon$ and $M_{\ref{stabthm}}$.  Apply Lemma \ref{lemma5} to $\frac{\delta}{2}$ to obtain $(d_0)_{\ref{lemma5}}$.  Apply Lemma \ref{lemma3} to $\epsilon$ to obtain $\beta$ and $(d_0)_{\ref{lemma3}}$.  Let $d_0=\min \{(d_0)_{\ref{lemma5}}, (d_0)_{\ref{lemma3}}\}$.  Apply Lemma \ref{lemma3} to $d=d_0\leq (d_0)_{\ref{lemma3}}$ and $\eta = \min \{ \gamma_{el}(d,3), \frac{1}{M_{\ref{stabthm}}} \}$ to obtain $M_{\ref{lemma3}}$.  Set $M=\max \{ CM(\eta, \frac{1}{\eta}), M_{\ref{lemma3}} \}$ and fix $n>M$.  Lemma \ref{lemma3} implies
\begin{eqnarray}\label{boundtouse}
\frac{|M_r(n)\setminus E_r(n,\epsilon, \eta, d)|}{m(r)^{n\choose 2}} \leq 2^{-\beta n^2}.
\end{eqnarray}
We now show $E_r(n,\epsilon, \eta, d)\subseteq D_r(n,\delta/2,\eta, d)$.  Suppose $G\in E_r(n,\epsilon, \eta, d)$.  We need to show that $Q_{\eta,d}(G) \neq \emptyset$ and for all $R\in Q_{\eta,d}(G)$, $R\in \tilde{C}_r^{\delta/2}(t)$ where $t=|V(R)|$.  As $n>CM(\eta, \frac{1}{\eta})$, we have $Q_{\eta,d}(G) \neq \emptyset$.  Suppose $R\in Q_{\eta,d}(G)$ and set $t=|V(R)|$.  By definition of $E_r(n,\epsilon, \eta, d)$, $R\in \tilde{E}_r(t,\epsilon)$.  Theorem \ref{stabthm} and our assumptions on $\eta$ imply that $R\in \tilde{C}^{\delta/2}_r(t)$, so $E_r(n,\epsilon, \eta, d)\subseteq D_r(n,\delta/2,\eta, d)$.  Lemma \ref{lemma5} implies $D_r(n,\delta/2,\eta, d)\subseteq C_r^{\delta}(n)$.  Combining these inclusions with (\ref{boundtouse}) we have that
$$
\frac{|M_r(n)\setminus C^{\delta}_r(n)|}{m(r)^{n\choose 2}} \leq 2^{-\beta n^2}.
$$
By Remark \ref{lowerbound}, $|M_r(n)|\geq m(r)^{n\choose 2}$, so 
$$
\frac{|M_r(n)\setminus C^{\delta}_r(n)|}{|M_r(n)|} \leq \frac{|M_r(n)\setminus C^{\delta}_r(n)|}{m(r)^{n\choose 2}} \leq 2^{-\beta n^2},
$$
which completes our proof of Theorem \ref{deltaclosethm}.
\qed

%************************************************************************
\section{Proof of Theorem \ref{aathm}}\label{provingtheorem1}

%************************************************************************
\setcounter{theorem}{0}
\numberwithin{theorem}{section}

In this section we prove Theorem \ref{aathm}, which says that for all even integers $r\geq 4$, almost all $G$ in $M_r(n)$ are in $C_r(n)$.  The outline of the proof of is as follows.  Given $\epsilon>0$ and integers $r,n\geq 3$, define 
\begin{align*}
A_r(n,\epsilon) &= \{ G\in M_r(n): \exists x\in [n]\textnormal{ such that for some } l\in[m(r)-2], |N^G_l(x)|\geq \epsilon n\},\\
A'_r(n,\epsilon) &= \{G\in M_r(n) \setminus A_r(n,\epsilon): \exists xy\in {[n]\choose 2} \textnormal{ with } d^G(x,y) \in [m(r)-2]\}.
\end{align*}

For all $\epsilon>0$, $n\in \mathbb{N}$, and even integers $r\geq 4$, we have that $M_r(n)=C_r(n) \cup A_r(n,\epsilon)\cup A'_r(n,\epsilon)$, and thus $M_r(n)\setminus C_r(n) \subseteq A_r(n,\epsilon)\cup A'_r(n,\epsilon)$.  We will show that when $r$ is even, there are $\epsilon>0$ and $\beta>0$ such that for large $n$, $|A_r(n,\epsilon)\cup A'_r(n,\epsilon)|\leq 2^{-\beta n}|M_r(n)|$, from which Theorem \ref{aathm} will follow.  We do this in two lemmas, one for each of the sets $A_r$ and $A_r'$ defined above.  The first lemma will apply to all $r\geq 3$, while the second will apply only to even $r\geq 4$.  

\begin{lemma}\label{boundA}
For all integers $r\geq 3$ and all $\epsilon>0$ there is $\beta>0$ and $M$ such that $n>M$ implies
\begin{eqnarray}\label{4}
|A_r(n,\epsilon)|\leq 2^{-\beta n^2}|C_r(n)|.
\end{eqnarray}
\end{lemma}
\begin{proof}
Let $r\geq 3$ be an integer and fix $\epsilon>0$.  By Remark \ref{lowerbound}, it suffices to find $\beta>0$ and $M$ such that $n>M$ implies 
$$
|A_r(n,\epsilon)|\leq 2^{-\beta n^2}m(r)^{{n\choose 2}}.
$$
Choose $T>0$ large enough so that $\frac{\epsilon^2 T^2}{8}-\frac{\epsilon T}{4}\geq 1$, then choose $0<\delta< \min \{\frac{1}{T}, \frac{\epsilon^2}{16}\}$.  Apply Theorem \ref{stabthm} to $\delta$ to obtain $\epsilon_{\ref{stabthm}}$ and $M_{\ref{stabthm}}$.  Apply Lemma \ref{lemma3} to $\epsilon_{\ref{stabthm}}$ to obtain $d_0$ and $\beta>0$.  Choose $d\leq d_0$ and $\eta< \min \{ \delta, \gamma_{el}(d,3), \frac{\epsilon}{2},d, \frac{1}{M_{\ref{stabthm}}}\}$.  Apply Lemma \ref{lemma3} to this $d$ and $\eta$ to obtain $M_{\ref{lemma3}}$.
Choose $M\geq \max\{ M_{\ref{lemma3}},CM(\frac{1}{\eta},\eta)\}$.  Lemma \ref{lemma3} implies that for all $n>M$,
$$
\frac{|M_r(n)\setminus E_r(n,\epsilon_{\ref{stabthm}},\eta, d)|}{m(r)^{n\choose 2}}\leq 2^{-\beta n^2}.
$$
Therefore, it suffices to prove that $n>M$ implies that $A_r(n, \epsilon)\subseteq M_r(n) \setminus E_r(n,\epsilon_{\ref{stabthm}}, \eta, d)$.  Fix $n>M$ and suppose for a contradiction that there is some $G\in A_r(n,\epsilon) \cap E_r(n,\epsilon_{\ref{stabthm}}, \eta, d)$.  Since $G\in A_r(n,\epsilon)$, there is $x\in [n]$ and $l\in [m(r)-2]$ such that $|N^G_l(x)|\geq \epsilon n$.  Because $n>CM(\frac{1}{\eta}, \eta)$, there is $R\in Q_{\eta,d}(G)$.  Also, $G\in E_r(n,\epsilon_{\ref{stabthm}}, \eta, d)$ implies that $W(R) \geq m(r)^{(1-\epsilon_{\ref{stabthm}}){t\choose 2}}$ where $t=|V(R)|$.  Then $t\geq \frac{1}{\eta}\geq M_{\ref{stabthm}}$ implies that there is $R'\in \tilde{C}_r(t)$ such that $|\Delta(R,R')|\leq \delta t^2$.  

Let $\mathcal{P}=\{V_1,\ldots, V_t\}$ be an $\eta$-regular partition for $G$ such that $R=R(G,\mathcal{P}, d)$, and define $\Sigma=\{ i\in [t]: |N^G_l(x)\cap V_i|\geq \frac{\epsilon}{2} |V_i|\}$.  We have that
\begin{align*}
\epsilon n \leq |N^G_l(x)|= \sum_{i\in \Sigma} |N_l^G(x)\cap V_i|+\sum_{i\notin \Sigma} |N_l^G(x)\cap V_i| \leq |\Sigma|\frac{n}{t} + (t-|\Sigma|)\frac{\epsilon}{2}\frac{n}{t}= |\Sigma|(1-\frac{\epsilon}{2})\frac{n}{t} + \frac{\epsilon n }{2}.
\end{align*}
Rearranging this, we obtain that $|\Sigma|\geq (\frac{\epsilon n }{2})/((1-\frac{\epsilon}{2})\frac{n}{t}) = \frac{\epsilon t}{2(1-\frac{\epsilon}{2})}\geq \frac{\epsilon t}{2}$.  Set 
$$
I = \{ ij \in E(\Sigma): (V_i,V_j)\textnormal{ is }\eta\textnormal{-regular for }G\textnormal{ and }c^R(ij) = c^{R'}(ij)\}.
$$
Applying that $\mathcal{P}$ is an $\eta$-regular partition for $G$, that $|\Delta(R,R')|\leq \delta t^2$, and that $\frac{\epsilon t}{2}\leq |\Sigma|$ yields 
\begin{eqnarray}\label{Psize}
|I|\geq { \frac{\epsilon t}{2}\choose 2} - \eta t^2 - \delta t^2 = t^2\Bigg(\frac{\epsilon^2}{4} - \eta - \delta \Bigg) - \frac{\epsilon t}{4} \geq  t^2\Bigg(\frac{\epsilon^2}{4} - 2\delta \Bigg)-\frac{\epsilon t}{4}, 
\end{eqnarray}
where the last inequality is because $\eta \leq \delta$.  By our assumptions on $\delta$ and because $t\geq \frac{1}{\delta}\geq T$, the right hand side of (\ref{Psize}) is at least $\frac{\epsilon^2 t}{8} - \frac{\epsilon t}{4}\geq 1$.  Thus $I\neq \emptyset$.

Take $ij\in I$ and let $W_i = N^G_l(x)\cap V_i$ and $W_j = N^G_l(x)\cap V_j$.  Since $\eta\leq \frac{\epsilon}{2}$ and $(V_i, V_j)$ is $\eta$-regular for $G$, we have $\rho^G_{r-1}(W_i,W_j)\geq \rho^G_{r-1}(V_i,V_j)-\eta$.  Because $c^R(ij)=c^{R'}(ij)$, we have that $r-1\in c^R(ij)$.  Therefore, by definition of $R$, $\rho^G_{r-1}(V_i,V_j)\geq d$, so $\rho^G_{r-1}(W_i,W_j)\geq d-\eta >0$, where the last inequality is by assumption on $\eta$.  Therefore, there is $(x_i,x_j)\in W_i\times W_j$ such that $d^G(x_i,x_j)=r-1$.  But now $d^G(x,x_i) = l$, $d^G(x,x_j)= l$, and $d^G(x_i,x_j)=r-1$ implies that $\{x,x_i,x_j\}$ is a violating triangle in $G$, a contradiction.  This finishes the proof that $A_r(n,\epsilon)\subseteq M_r(n)\setminus E_r(n,\epsilon_{\ref{stabthm}}, \eta, d)$, so we are done.  
\end{proof}

\begin{lemma}\label{A''lemma}
Let $r\geq 4$ be an even integer integer.  There are $\epsilon, \beta>0$ and $N$ such that $n>N$ implies
\begin{eqnarray}\label{A''inequality}
|A_r'(n,\epsilon)|\leq 2^{N^2-\beta n}|C_r(n)|.
\end{eqnarray}
\end{lemma}
\begin{proof}
All logs are base $2$.  Set $\beta=\frac{1}{2}(\log m(r)^2 - \log (m(r)^2-2))$ and choose $\epsilon>0$ small enough so that 
\begin{eqnarray}\label{epsilon}
2r(H(\epsilon)+\epsilon)-2\beta<-\frac{3\beta}{2}.
\end{eqnarray}
Given an integer $k$, set
\begin{align*}
F(k)&=\log {k\choose 2} + \log(m(r)-2) - 2\log(m(r)^2-2) +2rk(H(\epsilon)+\epsilon) \text{ and }\\
F'(k) &= F(k) +3\log m(r).
\end{align*}
By Corollary \ref{oddcor}, there is $n_0$ such that $n>n_0$ implies
\begin{eqnarray}\label{o(n^2)}
|M_r(n)|\leq 2^{(n-1)^2-\beta n}m(r)^{n\choose 2}=2^{(n-1)^2-\beta n}|C_r(n)|.
\end{eqnarray}
By (\ref{epsilon}) and definition of $F'(n)$, there is $n_1$ such that $n>n_1$ implies 
\begin{align}
F'(n) -2\beta n +5<-\beta n. \label{A''1}
\end{align}
Apply Lemma \ref{boundA} to $\epsilon$ to obtain $M_{\ref{boundA}}$ and $\beta_{\ref{boundA}}$.  Choose $N >\max\{M_{\ref{boundA}},n_0, n_1\}$ large enough so $\beta_{\ref{boundA}}(N-2)^2 >1$.  We show by induction that for all $n\geq N$, (\ref{A''inequality}) holds.  We begin with the base cases $n=N$ and $n=N+1$.  Combining (\ref{o(n^2)}) with the fact that for all $n$, $A'_r(n,\epsilon)\subseteq M_r(n)$ yields 
\begin{align*}
|A'_r(N,\epsilon)|&\leq |M_r(N)|\leq 2^{(N-1)^2- \beta N}|C_r(N)|< 2^{N^2-\beta N }|C_r(N)|\textnormal{ and }\\
|A'_r(N+1,\epsilon)|&\leq |M_r(N+1)|\leq 2^{N^2-\beta(N+1)}|C_r(N+1)|.
\end{align*}
Therefore (\ref{A''inequality}) holds for $n=N$ and $n=N+1$.  Suppose now $n\geq N+2$ and (\ref{A''inequality}) holds for all $m$ such that $N\leq m\leq n-1$.  We show it holds for $n$.  We can construct any element $G$ of $A'_{r}(n,\epsilon)$ as follows.
\begin{enumerate}[$\bullet$]
\item Choose a pair of elements $xy\in {[n]\choose 2}$.  There are ${n\choose 2}$ ways to do this.
\item Choose $d^G(x,y)\in [m(r)-2]$.  There are $m(r)-2$ ways to do this.
\item Put a structure on $[n]\setminus \{x,y\}$.  There are $|M_r(n-2)|$ ways to do this.
\item For each $l\in [m(r)-2]$, choose $N_l(x)$ and $N_l(y)$.  Since $G$ is not in $A_r(n,\epsilon)$, for each $l\in [m(r)-2]$, $\max \{|N_l(x)|,|N_l(y)|\} \leq \epsilon n$.  Therefore, there are at most $({n\choose \epsilon n}2^{\epsilon n})^{2(m(r)-2)}\leq 2^{2rn(H(\epsilon)+ \epsilon)}$ ways to do this.
\item For each $z\in [n]\setminus (\{x,y\}\cup \bigcup_{l=1}^{m(r)-2}N_l(x)\cup N_l(y))$, choose $d^G(x,z)$ and $d^G(y,z)$.  Note $(d^G(x,z),d^G(y,z))$ must be chosen from $[m(r)-1,r]\times [m(r)-1,r] \setminus \{(m(r)-1,r), (r,m(r)-1)\}$, so there are at most $m(r)^2-2$ choices.
\end{enumerate}
Combining all of this we obtain that $|A'_r(n,\epsilon)|$ is at most 
\begin{align}\label{A''4}
{n\choose 2}(m(r)-2)2^{2rn(H(\epsilon)+ \epsilon)} (m(r)^2-2)^{n-2}|M_r(n-2)|= 2^{F(n)}(m(r)^2-2)^n|M_r(n-2)|.
\end{align}
Because $M_r(n-2)\subseteq C_r(n-2)\cup A_{r}(n-2,\epsilon)\cup A'_{r}(n-2,\epsilon)$, 
$$
|M_r(n-2)|\leq |C_{r}(n-2)|+|A_{r}(n-2,\epsilon)|+|A'_{r}(n-2,\epsilon)|.
$$
Lemma \ref{boundA} implies $|A_{r}(n-2,\epsilon)| \leq |C_{r}(n-2)|2^{-\beta_{\ref{boundA}} (n-2)^2}$, and our induction hypothesis implies $|A'_{r}(n-2,\epsilon)|\leq |C_{r}(n-2)|2^{N^2-\beta(n-2)}.$  Remark \ref{lowerbound} implies $|C_r(n)|=m(r)^{2n-3}|C_r(n-2)|$.  Combining these facts with (\ref{A''4}), we obtain that 
\begin{align}
|A'_r(n,\epsilon)|&\leq 2^{F(n)}(m(r)^2-2)^n(1+ 2^{-\beta_{\ref{boundA}}(n-2)^2}+2^{N^2-\beta (n-2)})|C_r(n-2)|\nonumber \\
&= 2^{F(n)}(m(r)^2-2)^n m(r)^{-2n+3}(1+ 2^{-\beta_{\ref{boundA}}(n-2)^2}+2^{N^2-\beta (n-2)})|C_r(n)|\nonumber \\
&= 2^{F'(n) - 2\beta n}(1+ 2^{-\beta_{\ref{boundA}}(n-2)^2}+2^{N^2-\beta (n-2)}) |C_r(n)|. \label{A'}
\end{align}
By assumption on $N$, $-\beta_{\ref{boundA}}(n-2)^2 <-1$, so we have that
\begin{align*}
1+ 2^{-\beta_{\ref{boundA}}(n-2)^2}+2^{N^2-\beta (n-2)}\leq 2+2^{N^2-\beta (n-2)}\leq \begin{cases}
4& \textnormal{ if } N^2-\beta(n-2) \leq 1,\\
2(2^{N^2-\beta (n-2)})& \textnormal{ if } N^2-\beta(n-2) > 1.
 \end{cases}
\end{align*}
Combining this with (\ref{A'}) yields that 
\[ 
|A_r'(n,\epsilon)|\leq \begin{cases}
 2^{F'(n) - 2\beta n +2}|C_r(n)| & \text { if }N^2-\beta(n-2)\leq 1 \text{ and }\\
2^{F'(n)- 3\beta n +N^2+5}|C_r(n)| &\text { if }N^2-\beta(n-2)> 1.
\end{cases} \]
In both cases we have $|A'_r(n,\epsilon)|\leq 2^{N^2 +F'(n)-2\beta n +5}|C_r(n)|$, so by (\ref{A''1}), $|A'_r(n,\epsilon)|\leq 2^{N^2-\beta n}|C_r(n)|$.  This completes the induction.
\end{proof}

\noindent{\bf Proof of Theorem \ref{aathm}}.  Fix $r\geq 4$ an even integer.  Apply Lemma \ref{A''lemma} to obtain $\epsilon_{\ref{A''lemma}}$, $\beta_{\ref{A''lemma}}$ and $N_{\ref{A''lemma}}$. Apply Lemma \ref{boundA} to $\epsilon_{\ref{A''lemma}}$ to obtain $\beta_{\ref{boundA}}$ and $M_{\ref{boundA}}$.  Set $\epsilon =\epsilon_{\ref{A''lemma}}$ and $\beta = \frac{1}{2}\beta_{\ref{A''lemma}}$.  Let $M'$ be large enough so that $n>M'$ implies $2^{-\beta_{\ref{boundA}} n^2}+2^{N_{\ref{A''lemma}}^2-\beta_{\ref{A''lemma}} n}< 2^{-\beta n}$.  Set $M=\max \{M_{\ref{boundA}},N_{\ref{A''lemma}},M' \}$.  For all $n$, by definition, $M_r(n)\setminus C_r(n) \subseteq A_r(n,\epsilon) \cup A'_r(n,\epsilon)$.
Therefore, when $n>M$ our assumptions imply
\begin{align*}
|M_r(n)\setminus C_r(n)| &\leq |A_r(n,\epsilon)|+|A'_r(n,\epsilon)| \leq (2^{-\beta_{\ref{boundA}}n^2}+2^{N^2_{\ref{A''lemma}}-\beta_{\ref{A''lemma}}n})|C_r(n)| <2^{-\beta n}|C_r(n)|.
\end{align*}
Rearranging yields that $|C_r(n)|\geq |M_r(n)|(1-2^{-\beta n})$, as desired.
\qed

%************************************************************************
\section{Concluding remarks}\label{concludingremarks}

%************************************************************************
\setcounter{theorem}{0}
\numberwithin{theorem}{section}

%The behavior of $M_r(n)$ is harder to %understand in the case when $r$ is odd than %in the case when $r$ is even.

\bigskip
\noindent
$\bullet$
When $r$ is odd, the error term in Corollary \ref{oddcor} cannot be strengthened from $o(n^2)$ to $o(1)$ (or even to $O(n)$), as in Corollary \ref{evencor}.  This can be seen by constructing a large collection of elements of $M_r(n)$, which will show that $|M_r(n)|$ is at least $m(r)^{{n\choose 2}+\Omega(n\log_{m(r)}(n))}$.  Fix $n$ a sufficiently large integer.  Define a \emph{matching} to be a set $S\subseteq {[n]\choose 2}$ such that no two elements of $S$ have nonempty intersection.  Given a matching $S$, define $A(S)$ to be the set of simple complete $r$-graphs $G$ such that for each $xy\in S$, $d^G(x,y)=m(r)-1$ and for each $xy\in {[n]\choose 2}\setminus S$, $d^G(x,y) \in [m(r),r]$.  One can easily verify that for any matching $S$, no element of $A(S)$ contains a violating triangle, so $A(S)\subseteq M_r(n)$, and that given another matching $S'\neq S$, $A(S)\cap A(S')=\emptyset$.  Further, it is clear that that $|A(S)|=m(r)^{{n\choose 2}-|S|}$ and $|S|\leq \frac{n}{2}$, so $|A(S)|\geq m(r)^{{n\choose 2}-\frac{n}{2}}$.  Finally, note that there are at least $(\frac{n}{2})!$ distinct matchings on $[n]$.  This and Stirling's approximation yields that
$$
|M_r(n)|\geq (\frac{n}{2})!m(r)^{{n\choose 2}-\frac{n}{2}} =m(r)^{{n\choose 2}+\Omega(n\log_{m(r)} n)}.
$$
Combining this with Theorem \ref{oddcor}, the best bounds we have obtained for $|M_r(n)|$ are
$$
m(r)^{{n\choose 2}+\Omega(n\log n)}\leq |M_r(n)|\leq m(r)^{{n\choose 2}+o(n^2)}.
$$
We conjecture that in fact, $|M_r(n)|=m(r)^{{n\choose 2}+ \Theta(n\log n)}$. 

\bigskip

\noindent
$\bullet$ It is impossible to extend Theorem \ref{aathm} to the case when $r$ is odd. Indeed, one can show that 
$$
|C_r(n)|\leq (1-r^{-66r^2})|M_r(n)|.
$$
The proof of this (see the appendix) in fact shows that there is a $\mathcal{L}_r$-sentence $\psi$ such that for all $n$, $C_r(n)\subseteq \{G\in M_r(n): G\models \neg \psi\}$, and
\begin{eqnarray}\label{CR}
|C_r(n)|\leq r^{65r^2}|\{G\in M_r(n): G\models \psi\}|.
\end{eqnarray}

Suppose we knew that for some $\alpha>0$, $|C_r(n)|\geq \alpha |M_r(n)|$ for all sufficiently large $n$.  Then since for all $G\in C_r(n)$, $G\models \neg \psi$ we would know that 
$$
|\{G\in M_r(n): G\models \neg \psi\}|\geq \alpha |M_r(n)|.
$$
Dividing both sides of this by $|M_r(n)|$ gives us that $\mu^{M_r}(\neg \psi)\geq \alpha$, and therefore $\mu^{M_r}(\psi)\leq 1-\alpha$.  By dividing the quantities in (\ref{CR}) by $|M_r(n)|$, we obtain that $|C_r(n)|/|M_r(n)| \leq \mu^{M_r}(\psi){r^{65r^2}}$, and therefore $\alpha /r^{65r^2} \leq \mu^{M_r}(\psi)$.  Combining these inequalities, we would have that
$$
0<\frac{\alpha}{ r^{65r^2}} \leq \mu^{M_r}(\psi) \leq 1-\alpha<1,
$$
that is, $\mu^{M_r}(\psi) \notin \{0, 1\}$.  Therefore, if we could show such an $\alpha$ existed, we would know that $M_r$ had no labeled first-order $0$-$1$ law.  However, we do not know that such an $\alpha$ exists.  In fact it seems likely to the authors that instead, $\lim_{n\rightarrow \infty}|C_r(n)|/ |M_r(n)|=0$.

%************************************************************************
\section{Appendix}\label{appendix}

%************************************************************************
\setcounter{theorem}{0}
\numberwithin{theorem}{section}

\noindent{\bf Proof of Lemma \ref{sizelemma}}.   Given an integer $r\geq 3$, subsets $A,B,C\subseteq [r]$, and integers $x, y$, write $H_r(A,B,C,x,y)$ to mean $A,B,C,x,y$ satisfy the hypotheses of the lemma for $r$.  We show by induction on $r$ that for all $r\geq 3$, $A,B,C\subseteq [r]$, and $x,y\in \mathbb{N}$, $H_r(A,B,C,x,y)$ implies $A\times B\times C$ contains a violating triple.

\underline{Case $r=3$}:  Fix $A$, $B$, $C\subseteq [3]$, and integers $x,y$ such that $H_3(A,B,C,x,y)$.  As $m(3)=2$ and $3-m(3)=1$, we have $|A|= 3$, $x=1$, $|B|\geq 2$, $0\leq y\leq 1$, and $|C| \geq \max\{ 2-1-y+2, 1\}=\max \{3-y,2\}$.  If $y=0$, then $|B|=2$ and $|C|\geq 3-y=3$, contradicting that $|B|\geq |C|$.  Therefore, $y=1$, $|B|=3$, and $|C|\geq 2$.  This implies that $A=B=[3]$ and $C\cap\{1,3\}\neq \emptyset$, so either $(3,1,1)$ or $(1,1,3)$ is in $A\times B\times C$, and we are done.  

\underline{Case $r>3$}: 
Let $r>3$ and suppose by induction that the claim holds for all $3\leq r'<r$.  Fix $A,B,C\subseteq [r]$ and integers $x,y$ such that $H_r(A,B,C,x,y)$.  Notice this implies $x \geq y\geq 0$ and $x\geq 1$.  Suppose $A,B,C\subseteq [r-1]$.  Then
\[
|A|= m(r)+x= \begin{cases} 
m(r-1)+x+1 & \textnormal{ if }r\textnormal{ is even}\\
m(r-1)+x & \textnormal{ if }r\textnormal{ is odd},
\end{cases}
\]
\[
|B|= m(r)+y= \begin{cases} 
m(r-1)+y+1 & \textnormal{ if }r\textnormal{ is even}\\
m(r-1)+y & \textnormal{ if }r\textnormal{ is odd}
\end{cases}
\]
and
\[
|C|\geq \begin{cases}
\max \{m(r)-x-y,1\} =\max\{ m(r-1)-(x+1)-(y+1)+3,1\} & \text{if } r \text{ is even} \\
\max \{m(r)-x-y+2,1\} =\max\{ m(r-1)-x-y+2,1\} & \text{if } r \text{ is odd} .
\end{cases}
\]
Thus, $H_{r-1}(A,B,C,x,y)$ holds when $r$ is odd, and $H_{r-1}(A,B,C,x+1,y+1)$ holds when $r$ is even, so we are done by the induction hypothesis.  Assume now one of $A$, $B$, or $C$ contains $r$.  Let $a=\min A$, $b=\min B$, $c=\min C$, $a'=\max A$, $b'=\max B$, and $c'=\max C$.  Our assumptions imply that
\begin{eqnarray}\label{a}
a\leq r-|A|+1 = r-(m(r)+x) +1  = \begin{cases}
m(r)-1-x& \textnormal{ if }r\textnormal{ is even}\\
m(r)-x& \textnormal{ if }r\textnormal{ is odd},
\end{cases} 
\end{eqnarray}
and 
\[ b\leq r-|B|+1 \leq r-(m(r) +y) +1 = \begin{cases}
m(r)-1-y & \textnormal{ if }r\textnormal{ is even}\\
m(r)-y & \textnormal{ if }r\textnormal{ is odd}.
\end{cases} \] 
Thus
\[ a+b\leq \begin{cases}
m(r)-1-x+m(r)-1-y=r-x-y& \textnormal{ if }r\textnormal{ is even}\\
m(r)-x+m(r)-y = r-x-y+1 & \textnormal{ if }r\textnormal{ is odd}.
\end{cases} \] 
If 
\[ c'> \begin{cases}
r-x-y & \textnormal{ if }r\textnormal{ is even}\\
r-x-y+1 & \textnormal{ if }r\textnormal{ is odd},
\end{cases} \] 
then $(a,b,c)$ is a violating triple and we are done.  So assume 
\begin{eqnarray}c'\leq \begin{cases} \label{c'}
r-x-y & \textnormal{ if }r\textnormal{ is even}\\
r-x-y+1 & \textnormal{ if }r\textnormal{ is odd}.
\end{cases} 
\end{eqnarray}
Note that
\[ c\leq \begin{cases}
r-x-y -|C|+1 \leq r-x-y-(m(r)-x-y)+1 = m(r)-1& \textnormal{ if }r\textnormal{ is even}\\
r-x-y+1-|C|+1 \leq r-x-y+1-(m(r)-x-y+2) +1 = m(r)-1& \textnormal{ if }r\textnormal{ is odd}.
\end{cases} \] 
Therefore
\[ c+a \leq \begin{cases}
m(r)-1+ m(r)-1-x = r-x& \textnormal{ if }r\textnormal{ is even}\\
m(r)-1+ m(r)-x = r-x & \textnormal{ if }r\textnormal{ is odd}.
\end{cases} \] 
If $b'>r-x$, then $(a,b',c)$ is a violating triple and we are done.  So assume $b'\leq r-x$. Because $x\geq 1$, this implies $r\notin B$.  Further, 
\begin{eqnarray}\label{b}
b \leq \begin{cases}
r-x-(m(r)+y) +1= m(r)-x-y-1& \textnormal{ if }r\textnormal{ is even}\\
r-x-(m(r)+y) +1= m(r)-x-y & \textnormal{ if }r\textnormal{ is odd}.
\end{cases} 
\end{eqnarray}
Suppose $r\notin C$.  Then we must have that $a'=r\in A$.  Therefore,
\[a'-b \geq \begin{cases}
r-(m(r)-x-y-1) = m(r) +x+y -1& \textnormal{ if }r\textnormal{ is even}\\
r-(m(r)-x-y) = m(r)+x+y-1 & \textnormal{ if }r\textnormal{ is odd}.
\end{cases} \]
We now have $c\leq m(r)-1< m(r)+x+y-1 \leq a'-b$, so $(a',b,c)$ is a violating triple, and we are done.

Suppose now $c'=r\in C$.  By (\ref{c'}), this implies that $r$ is odd, $x=1$ and $y=0$.  By (\ref{b}), $b\leq m(r)-1$.  Therefore, 
$$
c'-b\geq r-(m(r)-1)=m(r)>m(r)-x,
$$
so by (\ref{a}), $(a,b,c')$ is a violating triple.  This completes the induction.
\qed
\medskip

\noindent{\bf Proof of Lemma \ref{lemmatriangle}.}
We proceed by induction on $r\geq 3$.  The base case $r=3$ can easily be verified.  Suppose now the claim holds for all $3\leq r'<r$.  Set $A'=A\cap[r-1]$, $B'=B\cap [r-1]$, and $C'=C\cap [r-1]$.  

Suppose that $r$ is odd.  If $A, B, C\subseteq [r-1]$, then because $|A|=|B|=|C|=m(r)=m(r-1)$, the induction hypothesis implies that $A=B=C=[m(r-1)-1,r-1]=[m(r)-1,r-1]$, i.e. case (2)(a) holds.  Suppose now one of $A$, $B$, or $C$ contain $r$.  By relabeling if necessary, we may assume $r\in A$.  Let $a'=r\in A$, $b=\min B$ and $c=\min C$.  Then $b\leq r-|B|+1 = m(r)$.  Therefore $c\geq a-b \geq r-m(r)=m(r)-1$, so $C\subseteq [m(r)-1,r]$.  Similarly, $c\leq r-|C|+1 = m(r)$, so $b\geq a'-c\geq r-m(r)=m(r)-1$ implies $B\subseteq [m(r)-1,r]$.  If $b=c=m(r)-1$, then $(a',b,c)$ is a violating triple, a contradiction.  Thus as most one of $b$ or $c$ is $m(r)-1$.  Therefore, by relabeling if necessary, we may assume $B\subseteq [m(r),r]$ and $C\subseteq [m(r)-1,r]$.  Recall that $|[m(r)-1,r]|=m(r)+1 = |B|+1$, so this implies that $B=[m(r),r]$. Let $a=\min A$.  Then $r\in B$ and $c\leq m(r)$ implies $a\geq r-m(r)= m(r)-1$, so $A\subseteq [m(r)-1,r]$.  If $C=[m(r),r]$, then we are done.  If $C\neq [m(r),r]$, then $c<m(r)$ implies $(m(r)-1,r,c)$ is a violating triple, so $m(r)-1\notin A$.  Thus $A\subseteq [m(r),r]$ and $|A|=|[m(r),r]|$ implies $A=[m(r),r]$ and we are done.

Suppose now that $r$ is even.  Note that $\min\{|A'|, |B'|,|C'|\} \geq m(r)-1$.  If two elements of the set $\{|A'|, |B'|,|C'|\}$ are strictly greater than $m(r)-1=m(r-1)$, then Lemma \ref{sizelemma} implies there is a violating triple in $A'\times B'\times C'$, a contradiction.  Therefore by relabeling if necessary, we may assume $|A'|=|B'|=m(r)-1$, so $r\in A\cap B$.  Let $a=\min A$, $b=\min B$, $c=\min C$ and note that $\max\{a,b,c\}\leq r-m(r)+1=m(r)-1$.  Now $(a,r,c)$ and $(r,b,c)$ cannot be violating triples, so 
\begin{align*}
a&\geq r-c\geq r-(m(r)-1)=m(r)-1, \\
c&\geq r-b\geq r-(m(r)-1)=m(r)-1\textnormal{ and }\\
b&\geq r-c\geq r-(m(r)-1)=m(r)-1.
\end{align*}
Thus, $A,B,C\subseteq [m(r)-1,r]$.  Since $|A|=|B|=|C|=|[m(r)-1,r]|$, this implies $A=B=C=[m(r)-1,r]$.
\qed

\subsection{Proof that when $r$ is odd, $C_r(n)$ is not almost all of $M_r(n)$.}
Fix $r\geq 3$ an odd integer for the rest of this section.  In this section we show that it is not the case that almost all elements of $M_r(n)$ are in $C_r(n)$ by constructing, for each integer $n\geq 4$, a map $f: C_r(n)\rightarrow M_r(n)\setminus C_r(n)$ which is at most $r^{65r^2}$-to-$1$.  This will imply that for all $n\geq 4$, 
\begin{eqnarray}\label{1.8(2)}
| C_r(n)|\leq (1-r^{-66r^2})|M_r(n)|.
\end{eqnarray}
We start with some preliminary definitions. Given an integer $n$ and $X, Y$ disjoint subsets of $[n]$, set $X<_*Y$ if and only if
\begin{enumerate}[(i)]
\item $|X|<|Y|$ or
\item $|X|=|Y|$ and $\min X< \min Y$.
\end{enumerate}

\begin{definition} Fix an integer $n\geq 3$ and $G\in M_r(n)$.
\begin{enumerate}
\item A set $X\subseteq [n]$ is a \emph{component} of $G$ if for all $xy\in {X\choose 2}$, there is a sequence $(z_1,\ldots, z_k)$ of distinct elements of $X$ such that $x=z_1$, $y=z_k$, and for each $1\leq i\leq k-1$, $d^G(z_i,z_{i+1})=m(r)-1$.
\item A \emph{component decomposition} of $G$ is a partition $X_1,\ldots, X_l$ of $[n]$ such that each $X_i$ is a component of $G$.  Note that there is a unique component decomposition of $G$, up to relabeling.
\item If $X_1,\ldots, X_l$ is the component decomposition of $G$ and $X_1<_*\ldots <_* X_l$, we say $X_1,\ldots, X_l$ is the \emph{canonically ordered component decomposition} (c.o.c.d.) of $G$.
\item A component $X$ of $G$ is \emph{large} if $|X|\geq 2r$.  Otherwise it is \emph{small}.
\item Suppose $X_1,\ldots, X_l$ is the c.o.c.d. of $G$.  The \emph{minimal large component} of $G$ is 
\[
ML(G)=\begin{cases}
\emptyset &\textnormal{ if }\max\{|X_1|,\ldots, |X_l|\}<2r,\\
X_i \textnormal{ where }i=\min \{ j\in [l]: |X_j|\geq 2r\} &\textnormal{ otherwise}.
\end{cases}
\]
\item $H$ is the simple complete $r$-graph with vertex set $[4]$ such that $d^H(1,3)=d^H(2,4)=r-1$, $d^{H}(1,4) = r$, and $d^{H}(1,2)= d^G(2,3)=d^H(3,4)=m(r)-1$.
\item A \emph{bad cycle in $G$} is a sequence $(z_1,\ldots, z_k)$ of distinct elements of $[n]$ such that for each $1\leq i\leq k-1$, $d^G(z_i,z_{i+1})=m(r)-1$ and $d^G(z_1,z_k)=r$.  Say $G$ \emph{contains a bad cycle} if there are  $z_1,\ldots, z_k\in [n]$ such that $(z_1,\ldots, z_k)$ is a bad cycle in $G$.
\end{enumerate}
\end{definition}  

\begin{lemma}\label{HnotinC_r}
$H\in M_r(4)$, and for any integers $n\geq k\geq 4$, if $G\in M_r(n)$ contains a bad cycle, then $G\in M_r(n)\setminus C_r(n)$.  In particular, if $G\in M_r(n)$ and $G$ contains a copy of $H$, then $G\notin C_r(n)$.
\end{lemma}
\begin{proof}
That $H$ contains no violating triangles and is therefore in $M_r(4)$ can be checked easily.  Suppose now $n\geq k\geq 4$, $G\in M_r(n)$, and $(y_1,\ldots, y_k)$ is a bad cycle in $G$.  Suppose for a contradiction that $G\in C_r(n)$.  Then there is a partition $\mathcal{P}=\{V_1,\ldots, V_t\}$ of $[n]$ such that for all $xy\in {[n]\choose 2}$,
\[ d^G(x,y) \in 
\begin{cases} 
[m(r),r] &\textnormal{ if }xy\in E(V_i, V_j) \textnormal{ some }1\leq i<j\leq t,\\
[m(r)-1,r-1]&\textnormal{ if }xy\in {V_i\choose 2} \textnormal{ some }1\leq i\leq t.
\end{cases}
\]
Note that for all $xy\in {[n]\choose 2}$, if $x$ and $y$ are in the same component of $G$, then they are in the same element of $\mathcal{P}$.  Fix $1\leq i\leq t$ such that $y_1\in V_i$.  Then $d^G(y_1,y_2)=\ldots =d^G(y_{k-1},y_k)=m(r)-1$ implies $y_k$ is in the same component of $G$ as $y_1$, so $y_k\in V_i$.  This implies $d^G(y_1,y_k) \in [m(r)-1,r-1]$.  Because $(y_1,\ldots, y_k)$ is a bad cycle in $G$, by definition, $d^G(y_1,y_k)=r$, a contradiction.  Since $H$ contains a bad cycle, it follows immediately that if $G\in M_r(n)$ contains a copy of $H$, then $G\notin C_r(n)$.
\end{proof}
Suppose $n$ is an integer and $G\in M_r(n)$.  Given $X\subseteq [n]$, let $G[X]$ denote the simple complete $r$-graph with vertex set $X$ such that for all $xy \in E(X)$, $d^{G[X]}(x,y)=d^G(x,y)$.  Set
\begin{align*}
D_1(n)=&\{G\in C_r(n):\textnormal{ the c.o.c.d. of }G\textnormal{ has at least }4\textnormal{ small components}\}, \\
D_2(n)=&\{G\in C_r(n)\setminus D_1(n): \textnormal{ if }\{y_1,\ldots, y_4\}\textnormal{ are the least four elements of }ML(G), \\&\textnormal{ then }G[ML(G)\setminus \{y_1,\ldots, y_4\}] \textnormal{ has at most }3\textnormal{ large components}\},\\
D_3(n)=&C_r(n)\setminus (D_1(n)\cup D_2(n)).
\end{align*}
We are now ready to define our map $f$.
\begin{definition}
Given $n\geq 4$ and $G\in C_r(n)$, define $f(G)$ to be the simple complete $r$-graph with vertex set $[n]$ satisfying the following, where $Y_1,\ldots, Y_u$ denotes the c.o.c.d. of $G$.
\begin{enumerate}
\item If $G\in D_1(n)$, set $Y=\bigcup_{i=1}^4 Y_i$, and for each $i\in [4]$, set $y_i=\min Y_i$.  Given $xy\in {[n]\choose 2}$, set
\[ d^{f(G)}(x,y)=
\begin{cases}
d^{H}(i,j) &\textnormal{ if }xy=y_iy_j \in {\{y_1,\ldots, y_4\}\choose 2},\\
r-1  &\textnormal{ if }xy\in {Y\choose 2}\setminus  {\{y_1,\ldots, y_4\}\choose 2},\\
d^G(x,y) &\textnormal{ otherwise}.
\end{cases}
\]

\item If $G\in D_2(n)$, let $s\in [4]$ be such that $Y_s=ML(G)$ and let $y_1<y_2<y_3< y_4$ be the least four elements of $Y_s$.  Set $Y=\bigcup_{i=1}^{s-1}Y_i$ and $Y_s'=Y_s\setminus \{y_1,\ldots, y_4\}$.  Given $xy\in {[n]\choose 2}$, set
\[ d^{f(G)}(x,y)=
\begin{cases}
d^{H}(i,j) &\textnormal{ if }xy=y_iy_j \in {\{y_1,\ldots, y_4\}\choose 2},\\
r  &\textnormal{ if }xy\in {Y\choose 2}\cup E(Y, \{y_1,\ldots, y_4\})\\
d^G(x,y)+1 &\textnormal{ if } xy=y_iz \textnormal{ for some }y_i\in \{y_1,\ldots, y_4\}\textnormal{ and }z\in Y_s',\\
d^G(x,y) &\textnormal{ otherwise}.
\end{cases}
\]
Note that any small component of $f(G)$ is either a singleton coming from $Y$, or is a small component of $f(G)[Y_s']$.  If $X$ is a small component of $f(G)[Y_s']$, then since $X$ and $\{y_1,\ldots, y_4\}$ were in the same component of $G$, there must be $x\in X$ and $y\in \{y_1,\ldots, y_4\}$ such that $d^G(x,y)=m(r)-1$, and thus, $d^{f(G)}(x,y) = m(r)$.  In particular, if $X=\{x\}$ is a singleton, then for some $y\in \{y_1,\ldots, y_4\}$, $d^{f(G)}(x,y) = m(r)$. On the other hand, if $X=\{x\}$ is a singleton coming from $Y$, then by construction, for all $y\in \{y_1,\ldots, y_4\}$, $d^{f(G)}(x,y) =r$.
\item If $G\in D_3(n)$, let $s\in [4]$ be such that let $Y_s=ML(G)$ and let $y_1<y_2<y_3< y_4$ be the least four elements of $Y_s$.  Set $Y=\bigcup_{i=1}^{s-1} Y_i$ and $Y_s'=Y_s\setminus\{y_1,\ldots, y_4\}$.  Let $Z^1,\ldots, Z^k$ be the large components of $G[Y_s']$ listed so that $Z^1<_*\ldots <_*Z^k$. Note that for each $ij\in {[k]\choose 2}$ and $xy\in E(Z^i,Z^j)$, because $Z^i$ and $Z^j$ are different components in $G[Y_s']$, $d^G(x,y)\neq m(r)-1$.  Since $Z^i$ and $Z^j$ are contained in the same component of $G$, we know $d^G(x,y)\in [m(r)-1,r-1]$.  Therefore we must have $d^G(x,y)\in [m(r),r-1]$.   Enumerate each $Z^i = \{z^i_1,\ldots, z^i_{|Z^i|}\}$ in increasing order.

We inductively build a sequence $i_1, \ldots, i_k$ with the following properties:
\begin{enumerate}[(i)]
\item For each $1\leq j\leq k$, $i_j\in [2r]$.
\item For each $2\leq j\leq k-1$, $d^G(z^{j-1}_{i_{j-1}}, z^{j}_{i_{j}}) = |i_j- i_{j+1}|\in [r]$.
\end{enumerate}
Set $i_1=i_2=1$,   $i_3=d^G(z^1_{i_1},z^2_{i_2})+1$.  Notice $1\leq i_3\leq (r-1)+1 =r$, so $i_3 \in [2r]$, and by construction, $i_3-i_2=d^G(z_{i_1}^1,z_{i_2}^2)\in [r]$, so (i) and (ii) are satisfied for $j=1,2$.  Suppose we've defined $i_1,\ldots, i_j$ for $2\leq j<k$ such that (i) and (ii) hold for $j-1$.  Set
\[
i_{j+1}=
\begin{cases}
i_j+d^G(z^{j-1}_{i_{j-1}}, z^j_{i_j})&\textnormal{ if }i_j\leq r, \\
i_j-d^G(z^{j-1}_{i_{j-1}}, z^j_{i_j})&\textnormal{ if }i_j> r.
\end{cases}
\] 
By the induction hypothesis, $i_j \in [2r]$, so by the above definition, if $i_j\leq r$, then $i_{j+1}\in [2,2r]$ and if $i_j>r$, then $i_{j+1} \in [1,2r-1]$.  In either case, $i_{j+1}\in [2r]$ so (i) is satisfied for $j+1$.  We also have that (ii) is satisfied by $j+1$ since by definition, 
$$
|i_{j}-i_{j+1}|=d^G(z^{j-1}_{i_{j-1}}, z^j_{i_j})\in [r].
$$
This completes the construction of $i_1,\ldots, i_k$.  Given $xy\in {[n]\choose 2}$, set
\[ d^{f(G)}(x,y)=
\begin{cases}
r  &\textnormal{ if }xy\in {Y\choose 2}\cup E(Y,\{y_1,\ldots, y_4\})\cup \{z^1_{i_1}z^k_{i_k}\} , \\
d^{H}(i,j) &\textnormal{ if }xy=y_iy_j \in {\{y_1,\ldots, y_4\}\choose 2},\\
m(r)-1 & \textnormal{ if }xy=z^j_{i_j}z^{j+1}_{i_{j+1}}\in \{z^1_{i_1}z^2_{i_2},\ldots, z^{k-1}_{i_{k-1}}z^k_{i_k}\},\\
d^G(x,y)+1 &\textnormal{ if } xy=y_iz \textnormal{ for some }y_i\in \{y_1,\ldots, y_4\}\textnormal{ and }z\in Y_s',\\
d^G(x,y) & \textnormal{ otherwise}.
\end{cases}
\]
\end{enumerate}
Note that the same remarks as above for the case when $G\in D_2(n)$ apply here.  That is, if $X=\{x\}$ is a singleton component of $f(G)$, then either $d^{f(G)}(x,y)=r$ for all $y\in \{y_1,\ldots, y_4\}$ in which case $x$ is an element in a small component of $G$, or there is $y\in \{y_1,\ldots, y_4\}$ such that $d^{f(G)}(x,y) =m(r)$, in which case $x$ is an element of $ML(G)\setminus \{y_1,\ldots, y_4\}$.
\end{definition}

\begin{lemma}
Let $n\geq 4$ be an integer and $G\in C_r(n)$.  Then $f(G)\in M_r(n)\setminus C_r(n)$.
\end{lemma}
\begin{proof}
By definition, $f(G)$ must contain a copy of $H$, so $f(G)$ is not in $C_r(n)$ by Lemma \ref{HnotinC_r}.  We now show $f(G)\in M_r(n)$. We leave the verification of the case when $G\in D_1(n)$ to the reader, since it requires only the simplest types of arguments which we show below for the other cases.  So assume $G\in D_2(n)\cup D_3(n)$.  Let $Y_1,\ldots, Y_u$ be the c.o.c.d. of $G$, let $s$ be such that $Y_s=ML(G)$, and let $y_1<\ldots<y_4$ be the least elements of $Y_s$.  Set $Y=\bigcup_{i=1}^{s-1}Y_i$ and $Y_s' = Y_s\setminus \{y_1,\ldots, y_4\}$.   It suffices to show that if $x,y,z\in [n]$ are pairwise distinct and $E(\{x,y,z\})\cap \Delta(G,f(G))\neq \emptyset$, then $\{x,y,z\}$ is not a violating triangle in $f(G)$, or equivalently, $(d^{f(G)}(x,y), d^{f(G)}(y,z), d^{f(G)}(x,z))$ is not a violating triple.  We consider only the cases where $\{x,y,z\}\subseteq Y_s$, as the rest of the cases are similar to these or trivial.  

Fix $x,y,z\in [n]$ pairwise distinct such that $E(\{x,y,z\})\cap \Delta(G,f(G))\neq \emptyset$, $\{x,y,z\}\subseteq Y_s$.  If $\{x,y,z\}\subseteq \{y_1,\ldots, y_4\}$, let $i,j,k\in [4]$ be such that $x=y_i$, $y=y_j$, $z=y_j$.  Then by definition of $f(G)$, $\{x,y,z\}$ is a violating triangle in $f(G)$ if and only if  $\{i,j,k\}$ is a violating triangle in $H$.  Since, by Lemma \ref{HnotinC_r}, $H$ contains no violating triangles, we are done.  If $x,y\in \{y_1,\ldots, y_4\}$ and $z\in Y_s'$ or if $x,y\in Y_s'$ and $z\in \{y_1,\ldots, y_4\}$, then by definition, 
$$
d^{f(G)}(x,z)=d^G(x,z)+1, \quad d^{f(G)}(y,z)=d^G(y,z)+1, \textnormal{ and }\quad d^{f(G)}(x,y)\in [m(r)-1,r].
$$
Because $x,y,z$ were in the same component of $G$, $d^{G}(x,z), d^{G}(y,z)\in [m(r)-1,r-1]$.  Therefore
$$
(d^{f(G)}(x,z), d^{f(G)}(y,z),d^{f(G)}(x,y)) \in [m(r),r]\times [m(r),r]\times [m(r)-1,r],
$$
which contains no violating triples.  Up to relabeling, this leaves us with the case where $\{x,y,z\}\subseteq Y_s'$.  This case is vacuous when $G\in D_2(n)$, because for $G\in D_2(n)$, $E(Y_s')\cap \Delta(G,f(G))=\emptyset$.  So we are left with the case when $G\in D_3(n)$ and $\{x,y,z\}\subseteq Y_s'$.  

Let $Z^1<_*\ldots <_*Z^k$ be the c.o.c.d. of $G[Y_s']$, and for $1\leq j\leq k$, let $z^j_{i_j}\in Z^j$ be as in the definition of $f(G)$.  We must have $E(\{x,y,z\})\cap \{z^1_{i_1}z^2_{i_2},\ldots, z^{k-1}_{i_{k-1}}z^{k}_{i_k}, z^1_{i_1}z^k_{i_k}\}\neq \emptyset$ since otherwise $E(\{x,y,z\})\cap \Delta(G,f(G))= \emptyset$.  Assume that $xy \in \{z^1_{i_1}z^2_{i_2},\ldots, z^{k-1}_{i_{k-1}}z^{k}_{i_k}, z^1_{i_1}z^k_{i_k}\}$, and note this implies $d^{f(G)}(x,y)\in \{r,m(r)-1\}$. 

If $z\in \{z^1_{i_1},\ldots, z^k_{i_k}\}$, and $xz, yz\notin \{z^1_{i_1}z^2_{i_2},\ldots, z^{k-1}_{i_{k-1}}z^{k}_{i_k}, z^1_{i_1}z^k_{i_k}\}$, then by definition of $f$, 
$$
d^{f(G)}(x,z)=d^G(x,z) \quad\textnormal{ and } \quad d^{f(G)}(y,z)=d^G(y,z).
$$
Because $z$ is the same component of $G$ as $x$ and $y$, $d^G(x,z), d^G(y,z) \in [m(r)-1,r-1]$.  Because $z$ is in a different component of $G[Y_s']$ than both $x$ and $y$, $d^{G}(x,z), d^{G}(y,z) \neq m(r)-1$.  Therefore 
\begin{align*}
d^G(x,z), d^G(y,z)& \in [m(r),r-1], \textnormal{ so }\\
(d^{f(G)}(x,z), d^{f(G)}(y,z),d^{f(G)}(x,y)) & \in [m(r),r-1]\times [m(r),r-1]\times  \{m(r)-1,r\} ,
\end{align*}
which contains no violating triples.  If $z\in \{z^1_{i_1},\ldots, z^k_{i_k}\}$ and $xz\in \{z^1_{i_1}z^2_{i_2},\ldots, z^{k-1}_{i_{k-1}}z^{k}_{i_k}, z^1_{i_1}z^k_{i_k}\}$, then since $k\geq 4$, this implies $yz \notin \{z^1_{i_1}z^2_{i_2},\ldots, z^{k-1}_{i_{k-1}}z^{k}_{i_k}, z^1_{i_1}z^k_{i_k}\}$.  By definition, $d^{f(G)}(x,z)\in \{m(r)-1,r\}$, and as above, because $y$ and $z$ are in the same component of $G$ but different components of $G[Y_s']$, $d^{f(G)}(y,z)=d^G(y,z) \in [m(r),r-1]$.  Therefore 
$$
(d^{f(G)}(x,z), d^{f(G)}(y,z), d^{f(G)}(x,y)) \in \{m(r)-1,r\}\times [m(r),r-1]\times \{m(r)-1,r\},
$$
which contains no violating triples.  Up to relabeling we have now covered the cases where $z\in \{z^1_{i_1},\ldots, z^k_{i_k}\}$, so assume $z\in Y'_s \setminus \{ z^1_{i_1},\ldots, z^k_{i_k}\}$.  Then by definition,
\begin{align*}
d^{f(G)}(x,z)=d^G(x,z), \text{ }d^{f(G)}(y,z)=d^G(y,z) \in [m(r)-1,r-1].
\end{align*}
If $z$ is in the same component of $G[Y_s']$ as $x$, then $y$ and $z$ are in the same component of $G$ but different components of $G[Y_s']$, so $d^{G}(y,z)\neq m(r)-1$.  Therefore we have that
$$
(d^{f(G)}(x,z), d^{f(G)}(y,z), d^{f(G)}(x,y)) \in  [m(r)-1,r-1]\times [m(r),r-1]\times \{m(r)-1,r\},
$$
which contains no violating triples.  A similar argument covers the case where $z$ is instead in the same component of $G[Y_s']$ as $y$.  If $z$ is in a different component of $G[Y_s']$ than $x$ and $y$, then 
\begin{align*}
 d^G(x,z), d^G(y,z) &\neq m(r)-1 \textnormal{ so }\\
(d^{f(G)}(x,z), d^{f(G)}(y,z), d^{f(G)}(x,y)) &\in [m(r),r-1]\times [m(r),r-1]\times \{m(r)-1,r\} ,
\end{align*}
which contains no violating triples.  This completes the proof.
\end{proof}

We will use the following lemmas.  Given $K\subseteq M_r(n)$, set $f^{-1}(K)=\{G\in C_r(n): f(G)\in K\}$.

\begin{lemma}\label{D_1lemma}
Let $n\geq 4$ be an integer.  For all $G\in f(D_1(n))$, there is $E\subseteq {[n]\choose 2}$ such that $|E|\leq {4+8r\choose 2}$ and for all $G'\in f^{-1}(G)\cap D_1(n)$, $\Delta(G,G')\subseteq E$.
\end{lemma}
\begin{proof}
Suppose $G\in f(D_1(n))$ and $X_1,\ldots, X_l$ is the c.o.c.d. of $G$.  Suppose $X_1,\ldots, X_s$ enumerate the components which are singletons and $X_w$ is the unique component such that $G[X_w]$ contains a copy of $H$.  Then by definition of $f$, for any $G'\in C_r(n)$, $f(G')=G$ implies
$$
\Delta(G,G')\subseteq E(X_w\cup \bigcup_{i=1}^s X_i).
$$
If $s\leq 8r$, then set $E=E(X_w\cup \bigcup_{i=1}^s X_i)$.  Since in this case, 
$$
|E(X_w\cup \bigcup_{i=1}^s X_i)|\leq {4+s \choose 2} \leq {4+8r\choose 2},
$$
we are done.  Assume now $s>8r$.  Let $G'\in f^{-1}(G)\cap D_1(n)$ and let $Y_1,\ldots, Y_u$ be the c.o.c.d. of $G'$.  For $i\in [4]$, let $\{y_i\}=\min Y_i$ and $Y=\bigcup_{i=1}^4 Y_i$.  By definition of $f$, $\Delta(G, G') \subseteq {Y\choose 2}$.  Note that for each $1\leq j\leq 4$, $Y_j$ has size at most $2r-1$, so $|Y|\leq 4(2r-1)<8r$.  Since $s>8r$, there is some $1\leq i\leq s$ such that $X_i\cap Y=\emptyset$.  Combining this with the fact that $\Delta(G, G') \subseteq {Y\choose 2}$, yields that 
$$
X_i \in \{Y_5,\ldots, Y_u\},
$$
say $X_i=Y_k$, some $5\leq k\leq u$.  Then $|Y_k|=1$ and $Y_k>_* Y_4>_*\ldots>_* Y_1$ implies by definition of $<_*$ that $|Y_4|=|Y_3|=|Y_2|=|Y_1|=1$.  Therefore $Y= \{y_1,\ldots, y_4\}=X_w$, and $\Delta(G',G)\subseteq {X_w\choose 2}$.  Setting $E={X_w\choose 2}$ we are done, as $|X_w|=4$.

\end{proof}

\begin{lemma}\label{D_2lemma}
Let $n\geq 4$ be an integer.  For all $G\in f(D_2(n))$, there are $G_1, \ldots, G_8\in D_2(n)$ and $E\subseteq {[n]\choose 2}$ such that $f(G_1)=\ldots= f(G_8)=G$, $|E|\leq {4+6r\choose 2}$, and for all $G'\in f^{-1}(G)\cap D_2(n)$, there is $1\leq t\leq 8$ such that $\Delta(G_t,G')\subseteq E$.
\end{lemma}
\begin{proof}
Suppose $G\in f(D_2(n))$ and $X_1,\ldots, X_l$ is the c.o.c.d. of $G$.  Let $t$ be such that $ML(G)=X_t$.  By definition of $f$, there is a unique index $1\leq w\leq l$ such that $G[X_w]$ consists of a copy of $H$.  There is also be a unique (possibly empty) sequence $1\leq i_1<\ldots< i_v<w$ with the following properties: 
\begin{itemize}
\item For each $1\leq j\leq v$, $X_{i_j}=\{x_{i_j}\}$ is a singleton, and
\item For each $1\leq j\leq v$, for each $y\in X_w$, $d^G(x_{i_j},y) =r$, and
\item For all $j\notin\{i_1,\ldots, i_v\}$, if $X_j=\{x_j\}$ is a singleton, then for some $y\in X_w$, $d^{G}(x,y)=m(r)$.  
\end{itemize}
Suppose $G'\in f^{-1}(G)\cap D_2(n)$.  Suppose $Y_1,\ldots, Y_u$ is the c.o.c.d. of $G'$ and $s$ is such that $ML(G')=Y_s$.  By definition of $f$ on $D_2(n)$, we must have that $X_w$ consists of the least $4$ elements of $Y_s$.  By the discussion following the definition of $f$ on $D_2(n)$, 
$$
\bigcup_{j=1}^v X_{i_j}=\bigcup_{i=1}^{s-1}Y_i,
$$
and the small components of $G'[Y_s\setminus X_w]$ are exactly the elements of $\{X_1,\ldots, X_{t-1}\}\setminus \{X_{i_1},\ldots, X_{i_v}, X_w\}$.  Notice that by definition of $D_2(n)$, $s\leq 4$, so 
\begin{align*}
\Bigg|\bigcup_{j=1}^v X_{i_j}\Bigg|=\Bigg|\bigcup_{i=1}^{s-1}Y_i\Bigg|\leq 3(2r-1) <6r.
\end{align*}
By definition of $f$, we have that 
\begin{align}\label{D_2}
\{X_t,\ldots, X_l\}=\{Y_{s+1},\ldots, Y_u\}\cup\{\text{the large components of }G'[Y_s\setminus X_w]\}.
\end{align}
If $X_i$ is a large component of $G'[Y_s\setminus X_w]$, then $|X_i|\leq |Y_s\setminus X_w|<|Y_s|\leq |Y_{s+1}|\leq |Y_u|$.  So by definition of $<_*$,
\begin{align}\label{D_2'}
X_i<_* Y_{s+1}<_*\ldots <_*Y_u.
\end{align}
By definition of $D_2(n)$ there are at most $3$ large components of $G'[Y_s\setminus X_w]$.  Combining this with (\ref{D_2}) and (\ref{D_2'}), we have that the large components of $G'[Y_s\setminus X_w]$ are contained in $\{X_t, X_{t+1}, X_{t+2}\}$ (where we let $X_i=\emptyset$ if $i>l$).
In sum, for any $G'\in f^{-1}(G)\cap D_2(n)$, we have the following.
\begin{enumerate}[(i)]
\item $X_w$ consists of the least $4$ elements of $ML(G')$,
\item $\bigcup_{j=1}^v X_{i_j}$ is the union of the small components of $G'$ and has size strictly less than $6r$,
\item The small components of $G'[ML(G')\setminus X_w]$ are the elements of $\{X_1,\ldots, X_{t-1}\}\setminus \{X_{i_1},\ldots, X_{i_v}, X_w\}$,
\item The set of large components of $G'[ML(G')\setminus X_w]$ is some subset $S$ of $\{X_t, X_{t+1}, X_{t+2}\}$.
\end{enumerate}
Set $E=E(X_w\cup \bigcup_{j=1}^v X_{i_j})$, and given $S\subseteq \{X_{t}, X_{t+1}, X_{t+2}\}$, set 
$$
X_S= \Bigg(\bigcup_{X_i\in S} X_i \Bigg)\cup  \Bigg(\bigcup_{j\in [t-1] \setminus \{i_1,\ldots, i_v\}}X_j\Bigg).
$$
Then (iii) and (iv) show that for all $G'\in f^{-1}(G)\cap D_2(n)$, there is $S\subseteq \{X_t, X_{t+1}, X_{t+2}\}$ such that $ML(G')=X_S$.  Moreover, given such a $G'$ and $S$, by definition of $f$ and (i)-(iv), 
\begin{enumerate}[$\bullet$]
\item $\Delta(G,G') \subseteq  E \cup E(X_w,X_S)$ and
\item For all $xy \in E(X_w, X_S), \text{ }d^{G'}(x,y)=d^G(x,y)-1$.
\end{enumerate}
Therefore, for all other $G''\in f^{-1}(G)\cap D_2(n)$ such that $ML(G'')=X_S$, we have that for all  $xy\in E(X_w,X_S)$, $d^{G''}(x,y)=d^G(x,y)-1=d^{G'}(x,y)$, so $\Delta(G',G'') \cap E(X_w,X_S)=\emptyset$.  This implies that
$$
\Delta(G',G'')\subseteq (\Delta(G',G)\cup \Delta(G'',G))\setminus E(X_w,X_S) \subseteq E.
$$

We now define $G_1,\ldots, G_8$.  Let $S_1,\ldots, S_{8}$ enumerate the subsets of $\{X_t,X_{t+1},X_{t+2}\}$.  For each $1\leq i\leq 8$, if there is $G'\in f^{-1}(G)\cap D_2(n)$ such that $ML(G')= X_{S_i}$, choose $G_i$ to be such a $G'$.  If no such $G'$ exists, choose $G_i$ to be any element of $D_2(n)$.  By what we've shown, for all $G'\in f^{-1}(G)\cap D_2(n)$, there is $1\leq i\leq 8$ such that $ML(G')= X_{S_i}$, and therefore $\Delta(G',G_i)\subseteq E$.  By (ii), $|\bigcup_{j=1}^v X_{i_j}|<6r$, so $|X_w \cup \bigcup_{j=1}^v X_{i_j}|< 4+6r$ and $|E|\leq {4+6r\choose 2}$.  This completes the proof.
\end{proof}

\begin{lemma}\label{D_3lemma}
Let $n\geq 4$ be an integer.  For all $G\in f(D_3(n))$, there is $G_1\in f^{-1}(G)\cap D_3(n)$ and $E\subseteq {[n]\choose 2}$ such that $|E|\leq {4+6r\choose 2}+2$, and for all $G'\in f^{-1}(G)\cap D_1(n)$, $\Delta(G_1,G')\subseteq E$.
\end{lemma}
\begin{proof}
Suppose $G\in f(D_3(n))$ and $X_1,\ldots, X_l$ is the c.o.c.d. of $G$.  By definition of $f$, there are exactly two indices $1\leq w<b\leq l$ such that $G[X_w]$ consists of a copy of $H$, and such that there is a sequence $(z^1,\ldots, z^k)$ which is a bad cycle in $G[X_b]$ of some length $k\geq 4$.  Let $B$ be the simple complete $r$-graph with vertex set $X_b$ such that for all $1\leq i\leq k-1$, $d^B(z^i, z^{i+1})=d^B(z^1,z^k)=r-1$, and for all other $xy\in E(X_b)$, $d^B(x,y)=d^G(x,y)$.  Then by definition of $f$, $B$ must have $k$ components, $Z^1,\ldots, Z^k$ such that for each $1\leq i\leq k$, $Z^i$ is a large component of $B$ containing $z^i$.  Moreover, we must have that either $Z^1<_*\ldots<_* Z^k$ or $Z^k<_*\ldots<_* Z^1$.  Because  $(z^1,\ldots, z^k)$ is a bad cycle if and only if $(z^k,\ldots, z^1)$ is a bad cycle, we can relabel $(z^1,\ldots, z^k)$ if necessary so that $Z^1<_* \ldots<_*Z^k$.  There is also be a unique (possibly empty) sequence $1\leq i_1<\ldots< i_v<w$ with the following properties: 
\begin{itemize}
\item For each $1\leq j\leq v$, $X_{i_j}=\{x_{i_j}\}$ is a singleton, and
\item For each $1\leq j\leq v$, for each $y\in X_w$, $d^G(x_{i_j},y) =r$, and
\item For all $j\notin\{i_1,\ldots, i_v\}$, if $X_j=\{x_j\}$ is a singleton, then for some $y\in X_w$, $d^{G}(x,y)=m(r)$.  
\end{itemize}
Suppose $G'\in f^{-1}(G)\cap D_3(n)$ and $Y_1,\ldots, Y_u$ is the c.o.c.d. of $G'$.  Let $s$ be such that $ML(G')=Y_s$.  The same arguments as in the case when $G\in D_2(n)$ imply that $X_w$ consists of the least $4$ elements of $Y_s$, 
$$
\bigcup_{j=1}^v X_{i_j}=\bigcup_{i=1}^{s-1}Y_i,
$$
 the small components of $G'[Y_s\setminus X_w]$ are exactly the elements of $\{X_1,\ldots, X_{t-1}\}\setminus \{X_{i_1},\ldots, X_{i_v}, X_w\}$, and $|\bigcup_{j=1}^v X_{i_j}| <6r$.  Further, by definition of $f$ we must have that $Z^1,\ldots, Z^k$ are the large components of $G'[Y_s\setminus X_w]$.  In sum, for any $G'\in f^{-1}(G)\cap D_2(n)$, we have the following.
\begin{enumerate}[(i)]
\item $X_w$ consists of the least $4$ elements of $ML(G')$,
\item $\bigcup_{j=1}^v X_{i_j}$ is the union of the small components of $G'$ and has size strictly less than $6r$,
\item $\{X_1,\ldots, X_{t-1}\}\setminus \{X_{i_1},\ldots, X_{i_v},X_w\}\}$ is the set of small components of $G'[ML(G')\setminus X_w]$,
\item $Z^1,\ldots, Z^k$ are the large components of $G'[ML(G')\setminus X_w]$, and $Z^1<_*\ldots <_* Z^k$.
\end{enumerate}
Set $X= \bigcup_{j\in [t-1]\setminus \{i_1,\ldots, i_v\}}X_j$ and $Z=\bigcup_{j=1}^k Z^j$, and note (iii) and (iv) imply that $ML(G')=X\cup Z$.  Define 
\begin{align*}
E_1=E(X_w\cup \bigcup_{j=1}^v X_{i_j})\text{ and } E_2=E(X_w, X\cup Z).
\end{align*}
Then for all $G'\in f^{-1}(G)\cap D_3(n)$, the definition of $f$ and (i)-(iv) imply that $\Delta(G,G')\subseteq E_1\cup E_2 \cup \{z^1z^2, z^2z^3,\ldots, z^1z^k\}$ and for all $xy\in E_2$, $d^{G'}(x,y)=d^G(x,y)-1$.  We now show that we can also recover the value of $d^{G'}(z^{j-1}, z^{j})$ for each $2\leq j\leq k-1$.  For each $1\leq j\leq k$, let $z^j_{j_1},\ldots, z^j_{j_{|Z^j|}}$ enumerate the elements of $Z^j$ in increasing order. Let $s_1, \ldots, s_k$ be the indices such that $(z^1_{s_1},\ldots, z^k_{s_k})=(z^1,\ldots, z^k)$.  By definition of $f$, for each $2\leq i\leq k-1$, $d^{G'}(z^{i-1}, z^{i})= |s_{i+1}-s_i|$.  We have now shown that for all $G', G''\in f^{-1}(G)\cap D_3(n)$, 
\begin{enumerate}[$\bullet$]
\item $\Delta(G,G')\cup \Delta(G,G'') \subseteq E_1\cup E_2 \cup \{z^1z^2, z^2z^3,\ldots, z^1z^k\}$,
\item For all $xy\in E_2$, $d^{G'}(x,y)=d^G(x,y)-1=d^{G''}(x,y)$, and
\item For all $z^iz^{i+1}\in \{z^1z^2,\ldots, z^{k-2}z^{k-1}\}$, $d^{G'}(x,y)= |s_{i+1}-s_i|=d^{G''}(x,y)$.
\end{enumerate}
Therefore,
$$
\Delta(G',G'')\subseteq (\Delta(G,G')\cup \Delta(G'',G))\setminus (E_2\cup \{z^1z^2,\ldots, z^{k-2}z^{k-1}\})\subseteq E_1\cup \{z^{k-1}z^k, z^1z^k\}.
$$
Set $E=E_1 \cup \{z^{k-1}z^k, z^1z^k\}$ and take $G_1$ to be any element of $f^{-1}(G)\cap D_3(n)$.  By (ii), $|\bigcup_{j=1}^v X_{i_j}|<6r$, so $|X_w \cup \bigcup_{j=1}^v X_{i_j}|< 4+6r$ and $|E|\leq {4+6r\choose 2}+2$.  This completes the proof.
\end{proof}

We now prove that for all $n\geq 4$, (\ref{1.8(2)}) holds.  Fix an integer $n\geq 4$ and $G\in f(C_r(n))$. Define $E_1,\ldots, E_{10}\subseteq {[n]\choose 2}$ and $G_1,\ldots, G_{10} \in C_r(n)$ as follows.  If $G\notin f(D_1(n))$, set $E_1=\emptyset$ and $G_1=G$.  Otherwise, let $G_1=G$ and let  $E_1\subseteq {[n]\choose 2}$ be as in Lemma \ref{D_1lemma}.  If $G\notin f(D_2(n))$, let $E_2=\ldots=E_9=\emptyset$ and $G_2=\ldots= G_8 = G$.  Otherwise let $E\subseteq {[n]\choose 2}$ and $G_2,\ldots, G_9 \in D_2(n)$ be as in Lemma \ref{D_2lemma}, and set $E_2=\ldots=E_9=E$.  If $G\notin f(D_3(n))$, let $E_{10}=\emptyset$ and $G_{10}=G$.  Otherwise, let $E_{10}\subseteq{[n]\choose 2}$ and $G_{10}\in D_3(n)$ be as in Lemma \ref{D_3lemma}.  Then Lemmas \ref{D_1lemma},  \ref{D_2lemma}, and \ref{D_3lemma} imply that
\begin{align*}
f^{-1}(G)\cap D_1(n) &\subseteq \{G'\in C_r(n): \Delta(G_1,G')\subseteq E_1\}, \\
f^{-1}(G)\cap D_2(n) &\subseteq \bigcup_{i=2}^9 \{G'\in C_r(n): \Delta(G_i,G')\subseteq E_i\}, \text{ and}\\
f^{-1}(G)\cap D_3(n) &\subseteq \{G'\in C_r(n): \Delta(G_{10},G')\subseteq E_{10}\}.
\end{align*}
Since $C_r(n)=D_1(n)\cup D_2(n) \cup D_3(n)$, we have that
\begin{eqnarray}\label{counter}
f^{-1}(G)\subseteq \bigcup_{i=1}^{10} \{G'\in C_r(n): \Delta(G_i,G')\subseteq E_i\}.  
\end{eqnarray}
For each $1\leq i\leq 10$, every element of $\{G'\in C_r(n): \Delta(G_i,G')\subseteq E_i\}$ can be constructed by starting with $G_i$, then changing the edges contained in $E_i$.  There are at most $r^{|E_i|}$ ways to do this, and for each $i$, $|E_i|\leq {4+8r\choose 2}\leq 64r^2$.  Therefore, for each $i$, $|\{G'\in C_r(n): \Delta(G_i,G')\subseteq E_i\}|\leq r^{64r^{2}}$.  Combining this with (\ref{counter}), we have that    
$$
|f^{-1}(G)|\leq 10r^{64r^2}\leq r^{65r^2}.
$$
Since $f(C_r(n))\subseteq M_r(n)\setminus C_r(n)$, this implies $|M_r(n)\setminus C_r(n)|\geq |f(C_r(n))|\geq\frac{|C_r(n)|}{r^{65r^2}}$.  Rearranging this yields that
$$
|C_r(n)|\leq \frac{r^{65r^2}}{r^{65r^2}+1}|M_r(n)|=\Bigg(1-\frac{1}{r^{65r^2}+1}\Bigg)|M_r(n)| < (1-r^{-66r^2})|M_r(n)|,
$$
as desired.

%\bibliography{science1}
%\bibliography{/Users/rickysellers/Desktop/Bibtex/science1.bib}
\bibliography{/Users/carolineamelia/Desktop/Bibtex/science1.bib}
\bibliographystyle{amsplain}
%\bibliography{science1}
\end{document}